\documentclass[12pt]{amsart}
\usepackage{amsfonts}
\usepackage{mathrsfs}
\usepackage{bbm}
\usepackage{amsfonts}
\usepackage{mathrsfs}
\usepackage{bbm}
\usepackage{amsfonts,amsthm}
\usepackage{amssymb,amsmath,graphicx}
\usepackage{cases}
\usepackage{bm}
\usepackage{float}
\usepackage[colorlinks=true]{hyperref}
\hypersetup{urlcolor=blue, citecolor=red}
\usepackage{hyperref}
\usepackage{enumerate}
\usepackage{cite}
\usepackage{amsfonts,amsthm}
\usepackage{amssymb,amsmath,graphicx}
\usepackage{cases}
\usepackage{float}
\usepackage[colorlinks=true]{hyperref}
\usepackage{cleveref}
\hypersetup{urlcolor=blue, citecolor=red}
\numberwithin{equation}{section}
\usepackage{cleveref}
\crefname{equation}{}{}

\newtheorem{theorem}{Theorem}[section]
 
\newtheorem{lemma}[theorem]{Lemma} 
\newtheorem{proposition}[theorem]{Proposition} 
 
\newtheorem{remark}[theorem]{Remark}

\theoremstyle{definition}

\def\XXint#1#2#3{{\setbox0=\hbox{$#1{#2#3}{\int}$ }
		\vcenter{\hbox{$#2#3$ }}\kern-.6\wd0}}

\def\qaq{\quad\text{and}\quad}

\everymath{\displaystyle}

\begin{document}
	\title[On the parametric Schr\"{o}dinger-Poisson system]{Remarks on the parametric Schr\"{o}dinger-Poisson system}
	
\author{Chen Huang}
\address{College of Science, University of Shanghai for Science and Technology, Shanghai, 200093, China}
\email{chenhuangmath111@163.com}
\author{Sihua Liang}
\address{College of Mathematics, Changchun Normal University, Changchun, Jilin, 130032, China}
\email{liangsihua@ccsfu.edu.cn}
\author{Lei Ma}
\address{College of Science, University of Shanghai for Science and Technology, Shanghai, 200093, China}
\email{leima@usst.edu.cn}
\author{Patrizia Pucci}
\address{Department of Mathematics and Computer Science, University of Perugia, Via Vanvitelli 1, 06123 Perugia, Italy}
\email{patrizia.pucci@unipg.it}
	\date{}
	\begin{abstract} The paper is concerned with the existence of solutions for the following Schr\"{o}dinger-Poisson system involving different potentials:
	\begin{equation*}
	\begin{cases}
			-\Delta u+V(x)u-\lambda \phi u=f(u)&\quad\text{in}~\mathbb R^3,\\
		-\Delta\phi=u^2&\quad\text{in}~\mathbb R^3.
	\end{cases}
\end{equation*}
 We first consider the case when the potential  $V$ is positive and radial, so that the mountain pass theorem could be applied. In the second case we assume that the potential $V$ is coercive and sign-changing, which means that the Schr\"{o}dinger operator $-\Delta +V$ is allowed to be indefinite. To deal with this more difficult case, using a local linking argument and the Morse theory, we
 prove that the system has a nontrivial solution. Furthermore, we also show the asymptotical behavior of this solution. Additionally, the proofs rely on new ideas regarding the solutions of the Poisson equation.
As a main novelty with respect to corresponding results in \cite{MR4527586, MR3148130, MR2810583}, we only assume that $f$ satisfies the super-linear growth condition at the origin.
We believe that the methodology developed here can be adapted to study related problems concerning the existence of solutions for the Schr\"{o}dinger-Poisson system.
\\[3mm]{\bf Keywords:} Schr\"{o}dinger-Poisson system; Perturbation problems; Mountain Pass theorem; Local linking; Non-trivial solutions.
		 	\\[3mm]{\bf Mathematics Subject Classification 2020:} 35A15, 35B38, 35J50
	\end{abstract}
	\maketitle

	\section{Introduction and Main Results}
This paper is concerned with the existence of nontrivial solutions of the nonlinear Schr\"{o}dinger-Poisson system 
\begin{equation}\label{eq1}
	\begin{cases}
			-\Delta u+V(x)u-\lambda \phi u=f(u)&\quad\text{in }\mathbb R^3,\\
		-\Delta\phi=u^2&\quad\text{in }\mathbb R^3,
	\end{cases}
\end{equation}
where $V$ and $f$ are given functions and $\lambda$ is a positive constant. To find the mountain pass solution, we only need the following assumptions:\\
${\bf (V)}$ $V\in C(\mathbb{R}^3)$ is radial, i.e., $V(x)=V(|x|)$, $x\in\mathbb R^3$, $\inf\limits_{\mathbb R^3} V(x)= V_0>0$.\\
${\bf (f)}$ $f\in C(\mathbb{R})$ and $\lim\limits_{t\rightarrow0}\dfrac{f(t)}{t}=0.$\\
System \eqref{eq1} arises when we focus on standing waves $\psi(t,x)=e^{-i\omega t/\sqrt{2m}}u(x)$ for the time-dependent Schr\"{o}dinger equation
$$i\sqrt{2m}\frac{\partial \psi}{\partial t}=-\Delta \psi+U(x)\psi-\lambda\phi \psi-g(|\psi|)\psi,\quad (t,x)\in \mathbb{R}^{+}\times \mathbb{R}^3 $$
with potential $\phi$ depending on the wave function $\psi$ via the Poisson equation
$$-\Delta\phi=|\psi|^{2} \quad \mbox{in }\ \mathbb{R}^3.$$
That is why we call system \eqref{eq1} the stationary Schr\"{o}dinger Poisson system. Indeed, in \eqref{eq1},
$$V(x) = U(x)-\omega\ \mbox{ and }\ f(t) = g(|t|)t.$$
The Schr\"{o}dinger-Poisson system  \eqref{eq1} has been extensively studied, especially when $\lambda<0$. In recent years, there has been an increasing attention to the existence of positive solutions, ground states, sign-changing solutions and multiplicity of solutions related to system \eqref{eq1}. Ruiz \cite{MR2230354} focused on the following problem
\begin{equation*}
	\begin{cases}
			-\Delta u+V(x)u-\lambda \phi u=|u|^{p-2}u &\quad\text{in }\mathbb R^3,\\
		-\Delta\phi=u^2 &\quad\text{in }\mathbb R^3,
	\end{cases}
\end{equation*}
and gave existence and nonexistence, depending on the parameters $p\in(2,6)$ and $\lambda<0$. In particular, if $\lambda\leq-1/4$, the author showed that $p = 3$ is a critical value for the existence of positive solutions. By using the concentration compactness principle, Azzollini
and Pomponio in \cite{MR2422637} proved the existence of a ground state solution of system \eqref{eq1}, when $f(u)=|u|^{p-2}u$ and $p\in(3,6)$. Let us also mention the papers \cite{MR3500305,MR4288156,MR4417277,MR4527586,MR4593001,MR4790968,MR4803693,MR2230354,MR2099569,MR1659454,MR2595202,MR2422637,MR2417922}.

From the above references, it can be seen that when $\lambda<0$, the existence of solutions to  system \eqref{eq1} is closely related to the growth of the nonlinear term (i.e., the range of value of $p$). However, when the parameter $\lambda > 0$, the situation changes. In fact, when $V(x)\equiv V_0>0$, Mugnai in \cite{MR2810583} proved that there exist infinitely many triples $(\lambda,u,\phi)\in\mathbb{R}^+\times H^1(\mathbb{R}^3)\times D^{1,2}(\mathbb{R}^3)$ solving system \eqref{eq1}. 
In \cite{MR2810583} the parameter $\lambda$ is a Lagrange multiplier and is part of the constructed solutions, 
and the nonlinearity $f$ in \cite{MR2810583} satisfies the following conditions:\\
({\bf $\tilde f_{1}$}): $f\in C(\mathbb{R})$ and $f(0)=0$;\\
({\bf $\tilde f_{2}$}): there are two positive constants $C_1$, $C_2$ such that $$|f(t)|\leq C_{1}|t|^{p-1}+C_{2}|t|^{q-1} \quad\text{ with}\ \ 2<p<q<6;$$
({\bf $\tilde f_{3}$}): there is a positive constant $\mu>2$,  such that $0\leq tf(t)\leq\mu F(t)$ with $F(t)=\displaystyle{\int_{0}^tf(s)}\text{d}s;$\\
({\bf $\tilde f_{4}$}): $f$ is an odd function in $\mathbb R$.

For $\lambda>0$ sufficiently large, Jeong and Seok in \cite{MR3148130} proved the existence of a radial solution under the conditions $V(x)\equiv V_0>0$, $({\bf f})$, and the following assumption:\\
${\bf (\tilde f)}$:
    $\varlimsup_{|t|\to\infty}|t|^{1-p}|f(t)|< \infty$ for some $p\in(2,6)$ uniformly for $x\in\mathbb{R}^3$.\\
As a consequence, this generalized the assumptions requested
by Mugnai in \cite{MR2810583}.

The method in \cite{MR3148130} relies on the non-degeneracy of the solutions of the limiting equation to prove the boundedness of the Palais Smale (PS) sequences of the energy functional  corresponding to system \eqref{eq1}. As a result, it removes the Ambrosetti-Rabinowitz (AR) condition requested by Mugnai in \cite{MR2810583}, retaining only the growth condition. This is to ensure that the energy functional  corresponding to the equation can be defined in a Hilbert space.

A natural question then arises: {\it Can the conditions at infinity on the nonlinear term $f$ be removed, i.e. ${\bf (\tilde f)}$?} This is one of the major innovative points of this paper. Another significant innovation is that our approach is completely different from that in  \cite{MR3148130}. We have proved an additional property of the solutions of the Poisson equation. Taking advantage of this property, we are able to prove the boundedness of the (PS) sequences of the energy functional corresponding to system \eqref{eq1}, without relying on the non-degeneracy of the original problem.

Before stating the first main result, we introduce some preliminaries.
As  working space, we consider the Hilbert space 
\begin{equation*}
	H^1_r(\mathbb R^3)=\{u\in H^1(\mathbb{R}^3)\,:\,u(x)=u(|x|)\},
\end{equation*}
with the norm
\begin{equation*}
	\|u\|_{H^1_r(\mathbb R^3)}=\left(\int_{\mathbb{R}^3}\left(|\nabla u|^2+V(x)u^2\right)\mbox{d}x\right)^{\frac{1}{2}}.
\end{equation*}
It is well-known that $H^1_r(\mathbb R^3)$ is compactly embedded in $L^p(\mathbb{R}^3)$ for all $p$, with $2<p<6$.

Let $\lambda=\epsilon^{-2}$, with $\epsilon>0$, and let
 \begin{equation*}
	\bar u=\frac{u}{\epsilon}.
	\end{equation*}
With abuse of the notations, we still label $\bar u$ by $u$. Then system \eqref{eq1} is equivalent to  
\begin{equation}\label{eq2}
	\begin{cases}
-\Delta u+V(x) u-\phi u=\frac{1}{\epsilon}f(\epsilon u)&\quad\text{in }\mathbb R^3,\\
-\Delta\phi= u^2&\quad\text{in }\mathbb R^3.
\end{cases}
\end{equation}
From now on, we shall consider system \eqref{eq2} 
instead of \eqref{eq1}. Theorem~1 in Section~2.2.1
of \cite{MR2597943} shows that the solution of
the  Poisson equation $\eqref{eq1}_2$ is
\begin{equation}\label{eq4}
	\phi_u(x)=\frac{1}{4\pi}\int_{\mathbb R^3}\frac{u^2(y)}{|x-y|}\text{d}y,\quad x\in\mathbb R^3.
\end{equation}
Moreover, Proposition 7 of \cite{MR3148130} gives that
\begin{equation}\label{phi_estimate}
	\int_{\mathbb{R}^3}|\nabla\phi_u|^2\text{d}x=\int_{\mathbb R^3}\phi_uu^2\text{d}x\leq C\|u\|_{H^1_r(\mathbb R^3)}^4.
\end{equation}
The subscript $u$ of $\phi$ denotes the solution of system \eqref{eq2}  corresponding to $u^2$. 

The constant $C$ here and in the rest of the paper may change from line to line as long as what their dependence is clear.

Now, we able to present our first result.
\begin{theorem}\label{thm} If assumptions ${\bf (V)}$ and ${\bf (f)}$ hold, then there exists a number $\lambda^*>0$ such that if $\lambda\geq \lambda^*$  system \eqref{eq1} admits a non-trivial weak solution $u$, that is
	\begin{equation*}
			\int_{\mathbb R^3}\left(\nabla u\nabla\psi +V(x)u\psi \right)\text{d}x-\lambda \int_{\mathbb R^3} \phi u\psi \text{d}x=\int_{\mathbb R^3}f(u)\psi \text{d}x\ \ ~\text{for~all}~\psi\in C_0^\infty(\mathbb R^3)
	\end{equation*}
and
	\begin{equation*}
	\int_{\mathbb R^3}\nabla \phi\nabla \eta \text{d}x=\int_{\mathbb R^3}u^2\eta \text{d}x\ \ ~\text{for~all}~\eta\in C_0^\infty(\mathbb R^3).
\end{equation*}
	\end{theorem}

\begin{remark}{\rm
In our approach, the following inequality
$$\int_{\mathbb R^3}\phi_uu^2\text{d}x=\int_{\mathbb R^3}|\nabla\phi_{u}|^2\text{d}x\geq\int_{\mathbb R^3}|u|^3\text{d}x-\frac{1}{4}\int_{\mathbb R^3}|\nabla u|^2\text{d}x$$
is crucial. It is worth emphasizing that this inequality depends solely on the properties of the Poisson equation itself. To the best of our knowledge, few papers have utilized this inequality to prove the boundedness of the (PS) sequences. To some extent, our work  improves the related results in \cite{MR2810583,MR3148130} mentioned above.}
\end{remark}

We emphasize that in Theorem \ref{thm}, we
only considered the case where the Schr\"{o}dinger operator $-\Delta+V$ is positive definite. In this
situation the zero function $u={\bf 0}$ is a local minimizer of 
\begin{equation}\label{IDEF}
I(u)=\frac{1}{2}\int_{\mathbb R^3}\left(|\nabla u|^2+V(x)u^2\right)\text{d}x-\frac{1}{4}\int_{\mathbb R^3}\phi_{u} u^2\text{d}x-\int_{\mathbb R^3}F(u)\text{d}x,
\end{equation}
with $\phi_u$ given in \eqref{eq4}, and we may apply the mountain pass theorem.

However, if we are interested in standing waves $\psi(t,x)=e^{-i\omega t/\sqrt{2m}}u(x)$ with large $\omega$, then $V(x)$ may be negative somewhere and $u=  0$ is no longer a local minimizer of $I$. When $\lambda<0$, the corresponding indefinite
situation is studied in Liu and Wu \cite{MR3656292}.

It is interesting to note that unlike many works in this field, the second result in this paper is when $\lambda>0$ and the potential $V$ is sign-changing. We present our assumptions on the potential $V$ and the nonlinearity $f$.\\
${\bf(V_1)}$: $V\in C(\mathbb{R}^3)$ is bounded from below and, $\mu(V^{-1}(-\infty,M]))<\infty$ for all $M>0$, where $\mu$ is the Lebesgue measure on $\mathbb{R}^3$;\\
${\bf(V_2)}$: $0$ is not an eigenvalue of $-\Delta+V(x)$.\\
${\bf (f)}$: $f\in C(\mathbb{R})$ and
$\displaystyle{\lim\limits_{t\to0}\dfrac{f(t)}{t}}=0;$\\
${\bf (f')}$: there exists $\delta_0>0$ such that
$f(t)t\geq 0$ for all $t$, with $|t|\leq\delta_0.$

Inspired by \cite{MR3656292,MR3843290} and combining a local linking argument of Li and Willem in \cite{MR1312028} 
together with the infinite-dimensional Morse theorem in \cite{MR1196690}, we obtain the second main result.

\begin{theorem}\label{thm2}
 If assumptions ${\bf (V_1)}$, ${\bf (V_2)}$, ${\bf (f)}$ and ${\bf (f')}$ hold, then there exists  a number $\lambda^{**}>0$ such that for any $\lambda\ge \lambda^{**}$, system \eqref{eq1} has a nontrivial solution $u$, that is
	\begin{equation*}
			\int_{\mathbb R^3}\left(\nabla u\nabla\psi +V(x)u\psi \right)\text{d}x-\lambda \int_{\mathbb R^3} \phi u\psi \text{d}x=\int_{\mathbb R^3}f(u)\psi \text{d}x\ \ ~\text{for~all}~\psi\in C_0^\infty(\mathbb R^3)
	\end{equation*}
and
	\begin{equation*}
	\int_{\mathbb R^3}\nabla \phi\nabla \eta \text{d}x=\int_{\mathbb R^3}u^2\eta \text{d}x\ \ ~\text{for~all}~\eta\in C_0^\infty(\mathbb R^3).
\end{equation*}
\end{theorem}

\begin{remark}{\rm
Theorem \ref{thm2} can be understood in another way.
Assumption ${\bf (V_1)}$ allows us to choose $m>0$ such that
\begin{equation}\label{tilde_V}\tilde V(x)=V(x)+m>1 
\quad\mbox{for all }x\in\mathbb{R}^3.
\end{equation}
Hence, system \eqref{eq1} is equivalent to  
\begin{equation*}
	\begin{cases}
			-\Delta u+\tilde V(x)u-\lambda \phi u=\tilde f(u) &\quad\text{in }\mathbb R^3,\quad\tilde f(u)=f(u)+mu,\\
		-\Delta\phi=u^2 &\quad\text{in }\mathbb R^3.
	\end{cases}
\end{equation*}
A straightforward computation shows that the potential $\tilde V(x)$ is positive and coercive, and $\tilde f(u)$ has a local linear growth, in other words,
$$\tilde f(t)\in C(\mathbb{R})\ \mbox{ and }\ \lim\limits_{t\rightarrow0}\frac{\tilde f(t)}{t}=m>0.$$
Thus the essential difference between Theorem \ref{thm} and Theorem \ref{thm2} lies in the different behavior of the nonlinearity at the origin.}
\end{remark}

\begin{remark}{\rm
The innovations of this paper are as follows:\\
$(1)$ In terms of results, through Theorem \ref{thm} and Theorem \ref{thm2}, we show that when the parameter $\lambda$ is sufficiently large, the existence of solutions to the system depends only on the state of the nonlinear term at the origin.\\
$(2)$ In terms of proof techniques, we propose a completely new method. Additionally, by means of new properties on the solutions to the Poisson equation, we obtain the boundedness of the (PS) sequences of the modified functionals. To the best of our knowledge, this phenomenon is also a new discovery.}
\end{remark}

We also study the asymptotical behavior of non-trivial solution $u^\lambda$ with respect to $\lambda$.

\begin{theorem}\label{thm3} Let ${\bf (V)}$ and ${\bf (f)}$
be satisfied and let $u^\lambda$ be a non-trivial weak solutions of system \eqref{eq1}. Then, as $\lambda\to\infty$,
\begin{equation*}
	u^{\lambda}\to 0~\text{strongly~in}~H^{1}_{r}(\mathbb{R}^3),\quad
	\sqrt{\lambda}u^{\lambda}\to u_0~\text{strongly~in}~H^{1}_{r}(\mathbb{R}^3),
\end{equation*}		
where $u_0$ is a non-trivial weak solution of the following Choquard equation
\begin{equation*}
			-\Delta u+V(x)u=\frac{1}{4\pi}\int_{\mathbb R^3}\frac{u^2(y)}{|x-y|}\text{d}y~u\ \ ~\text{in}~\mathbb{R}^3.
\end{equation*}
\end{theorem}

The paper is organized as follows. In Section 2, we propose a new modified problem and we show the close relationship between the solutions of the modified problem and the solutions of the original problem \eqref{eq2}. In Section 3, we prove that the modified functional $I_{\epsilon}$ has critical points
in $H^1_r(\mathbb R^3)$ for $\epsilon>0$ small enough, and when the parameter $\lambda$ is sufficiently large, these critical points is exactly the weak solutions of the original problem 
\eqref{eq2}. In Section 4, we prove that similar results in the case of sign-changing potentials $V$. In Section 5, we give the proof of the asymptotical behavior of non-trivial solutions.

\section{The modified problem}\label{section_2}
Since the condition ${\bf (f)}$ only involves the information at the origin, a truncation of $F$ is introduced such that the corresponding functional is well-defined in the working space. It is worth emphasizing that this modification method is new and specially constructed based on the system \eqref{eq2}, as stated in 
Theorem~\ref{thm3}.

Write
\begin{equation*}
	d(t)\begin{cases}=1,&\quad\text{if } |t|<1,\\
		\in(0,1),&\quad\text{if } 1\leq|t|\leq 2,\\
		=0,&\quad\text{if } |t|> 2,
	\end{cases}
\end{equation*}
and $d(t)\in C_0^{\infty}(\mathbb{R})$ satisfying $-C(d)\leq d'(t)t\leq0$, where $C(d)$ is a positive constant. The truncation about $F$ is defined as
\begin{equation}\label{eq5}
	\tilde F(t)=d\left(\frac{t}{\sqrt{\epsilon}}\right)F(t)+\left(1-d\left(\frac{t}{\sqrt{\epsilon}}\right)\right)|t|^3,
\end{equation}
where $\epsilon\in(0,1)$ is a constant.
Then the directly calculation gives that
\begin{equation*}
\begin{split}
	|\tilde F'(t)|&\leq\left|d'(\frac{t}{\sqrt{\epsilon}}) F(t)\frac{1}{\sqrt{\epsilon}}\right|+\left|d(\frac{t}{\sqrt{\epsilon}})f(t)\right|+\left|d'(\frac{t}{\sqrt{\epsilon}})\frac{1}{\sqrt{\epsilon}}|t|^3\right|+\left|3(1-d(\frac{t}{\sqrt{\epsilon}}))t^{2}\right|\\
	&\leq\left|d'(\frac{t}{\sqrt{\epsilon}}) \frac{t}{\sqrt{\epsilon}}\right|\left|\frac{F(t)}{t}\right|+\left|d(\frac{t}{\sqrt{\epsilon}})f(t)\right|+\left|d'(\frac{t}{\sqrt{\epsilon}})\frac{t}{\sqrt{\epsilon}}\right|t^{2}
+\left|3(1-d(\frac{t}{\sqrt{\epsilon}}))t^{2}\right|.
	\end{split}
\end{equation*}
For any $\delta>0$, it follows from
 $-C(d)\leq d'(t)t\leq0$ and  ${\bf (f)}$ that there is a constant $C_\delta$ depending on $\delta$ such that
\begin{equation}\label{eq6}
	|\tilde F(t)|\leq \delta t^2+C_\delta |t|^3\qaq|\tilde F'(t)|\leq\delta |t|+C_\delta |t|^2,
\end{equation}
where $C_\delta$ is a constant independent of $\epsilon$.

According to system \eqref{eq5}, the following modified problem is considered instead of system \eqref{eq2}
\begin{equation}\label{eq7}
	\begin{cases}
		-\Delta u+V(x) u-\phi u=\frac{1}{\epsilon}\tilde F'(\epsilon u) &\quad\text{in }\mathbb R^3,\\
		-\Delta\phi= u^2 &\quad\text{in }\mathbb R^3.
	\end{cases}
\end{equation}
Therefore, the corresponding functional is defined as
\begin{equation}\label{I_epsilon_1}
	\text{$I_{\epsilon}(u)$}=\frac{1}{2}\int_{\mathbb R^3}\left(|\nabla u|^2+V(x)u^2\right)\text{d}x-\frac{1}{4}\int_{\mathbb R^3}\phi_{u} u^2\text{d}x-\frac{1}{\epsilon^2}\int_{\mathbb R^3}\tilde F(\epsilon u)\text{d}x.
\end{equation}
The directly calculation gives that $I_{\epsilon}(u)$ is a $C^1$ functional on $H^{1}_{r}(\mathbb{R}^3)$ and for $u,v\in H^1_r(\mathbb R^3)$, one has
\begin{equation*}
\langle I_{\epsilon}'(u),v\rangle=\int_{\mathbb R^3}\left(\nabla u\nabla v+V(x)uv\right)\text{d}x-\int_{\mathbb R^3}\phi_uuv\text{d}x-\frac{1}{\epsilon}\int_{\mathbb R^3}\tilde F'(\epsilon u)v\text{d}x.
\end{equation*}
\section{The existence of parameter-dependent Schr\"{o}dinger-Poisson system }

In this section, the non-trivial solution of system \eqref{eq1} is obtained and give the proof of Theorem \ref{thm}. Firstly, the mountain pass theorem is applied to prove the existence of critical points for the modified functional $I_{\epsilon}$. Secondly, the truncation of $\tilde F$ is removed when $\epsilon>0$ is sufficiently small by Moser iteration. Finally,
the mountain pass critical point of the modified functional $I_{\epsilon}$ is actually the solution to the system \eqref{eq2}.

\begin{lemma}\label{lem1}
	\text{$I_{\epsilon}(u)$} possesses the following properties:
\begin{flalign}
			\label{i}
			\tag{{$i$}}
			&\text{There are constants $\rho>0$ and $\beta>0$ independent of $\epsilon$ such that}&
		\end{flalign}
	$$I_{\epsilon}(u)\geq\beta>0\text{ for $u\in S_\rho=\{u:\|u\|_{H^1_r(\mathbb R^3)}=\rho\}$}.$$
\begin{flalign}
	\label{ii}
	\tag{{$ii$}}
	&\text{There exists a function $\zeta\in H^{1}_{r}(\mathbb{R}^3)$ such that $\|\zeta\|_{H^1_r(\mathbb R^3)}>\rho$ and $I_{\epsilon}(\zeta)<0$.}&
\end{flalign}
\end{lemma}
\begin{proof}
	It follows from \eqref{phi_estimate} and \eqref{eq6} that
\begin{equation*}
	\begin{split}
	I_{\epsilon}(u)&=\frac{1}{2}\|u\|^2_{H^1_r(\mathbb R^3)}-\frac{1}{4}\int_{\mathbb R^3}\phi_u u^2\text{d}x-\frac{1}{\epsilon^2}\int_{\mathbb R^3}\tilde F(\epsilon u)\text{d}x\\
	&\geq\frac{1}{2}\|u\|^2_{H^1_r(\mathbb R^3)}-\frac{C}{4}\|u\|^4_{H^1_r(\mathbb R^3)}-\frac{1}{\epsilon^2}\int_{\mathbb R^3}\left(\delta |\epsilon u|^2+C_\delta|\epsilon u|^3\right)\text{d}x.\\
\end{split}
\end{equation*}
Choosing $\delta>0$ small, one gives
 \begin{equation*}
	\begin{split}
		I_{\epsilon}(u)\geq \frac{1}{3}\|u\|^2_{H^1_r(\mathbb R^3)}-\frac{C}{4}\|u\|^4_{H^1_r(\mathbb R^3)}-C\|u\|^3_{H^1_r(\mathbb R^3)}
	\end{split},
\end{equation*}
where the Sobolev inequality has been used. This implies  \eqref{i}.

For \eqref{ii}, take a $h\in C_0^\infty(\mathbb{R}^3)$, $h\neq0$ and let $s$ be a constant large enough. Substituting $sh$ in $I_\epsilon$ yields
\begin{equation*}
\begin{split}
	I_{\epsilon}(sh)&\leq\frac{s^2}{2}\|h\|^{2}_{H^1_r(\mathbb R^3)}-\frac{s^4}{4}\int_{\mathbb R^3}\phi_hh^2\text{d}x+\frac{1}{\epsilon^2}\int_{\mathbb R^3}\left(\delta \epsilon^2s^2h^2+C_\delta \epsilon^3s^3|h|^3\right)\text{d}x\\
	&\leq\frac{s^2}{2}\|h\|^{2}_{H^1_r(\mathbb R^3)}-\frac{s^4}{4}\int_{\mathbb R^3}\phi_hh^2\text{d}x+\int_{\mathbb R^3}\left(\delta s^2h^2+C_\delta s^3|h|^3\right)\text{d}x.
\end{split}
\end{equation*}
Thanks to $\int_{\mathbb{R}^3}\phi_{h}h^{2}\text{d}x\geq0$, then $I_{\epsilon}(sh)\rightarrow-\infty$ as $s\rightarrow\infty.$ Hence, the lemma is completed with $\zeta=sh$.
\end{proof}
According to Lemma \ref{lem1}, it follows from the mountain pass theorem in Ambrosetti and Rabinowitz \cite{MR370183} that there is a sequence $\{u_n\}_{n=1}^\infty\subset H^1_r(\mathbb R^3)$ such that
\begin{equation}\label{eq8}
	I_{\epsilon}(u_n)\rightarrow c \qaq I_{\epsilon}'(u_n)\rightarrow 0\ \text{ as }n\rightarrow\infty
\end{equation}
with $c$ be a positive constant.
\begin{lemma}\label{lem2}
Let $\{u_n\}_{n=1}^\infty\subset H^1_r(\mathbb R^3)$ satisfies \eqref{eq8}. There exists a constant $\epsilon_*\in(0,1)$ such that if $0<\epsilon<\epsilon_*$, then $u_n\rightarrow u\text{ strongly in }H^1_r(\mathbb R^3)$ with $u\in H^1_r(\mathbb R^3)$.
\end{lemma}
\begin{proof}
	Let $\mu$ be a positive constant to be determined later. It follows from \eqref{eq6} that
	\begin{equation}\label{eq9}
		\begin{split}
		&\quad I_{\epsilon}(u_n)-\frac{1}{\mu}\langle I_{\epsilon}'(u_n),u_n\rangle\\
		&=\left(\frac{1}{2}-\frac{1}{\mu}\right)\int_{\mathbb R^3}\left(|\nabla u_n|^2+V(x)u_n^2\right)\text{d}x+\left(\frac{1}{\mu}-\frac{1}{4}\right)\int_{\mathbb R^3}\phi_{u_n}u_n^2\text{d}x-\frac{1}{\epsilon^2}\int_{\mathbb R^3}\tilde F(\epsilon u_n)\text{d}x\\
		&\quad\quad-\frac{1}{\mu\epsilon^2}\int_{\mathbb R^3}\tilde F'(\epsilon u_n)\epsilon u_n\text{d}x\\
		&\geq\left(\frac{1}{2}-\frac{1}{\mu}\right)\int_{\mathbb R^3}\left(|\nabla u_n|^2+V(x)u_n^2\right)\text{d}x+\left(\frac{1}{\mu}-\frac{1}{4}\right)\int_{\mathbb R^3}\phi_{u_n}u_n^2\text{d}x\\
		&\quad\quad-\left(1+\frac{1}{\mu}\right)\int_{\mathbb R^3}\left(\delta |u_n|^2+\epsilon C_\delta |u_n|^3\right)\text{d}x.
	\end{split}\end{equation}
Multiplying $|u_n|$ on both side of $\eqref{eq1}_2$ gives
\begin{equation}\label{eq10}
	\int_{\mathbb R^3}|u_n|^3\text{d}x=\int_{\mathbb R^3}{\nabla\phi_{u_n}}\nabla(|u_n|)\text{d}x\leq\int_{\mathbb R^3}|\nabla\phi_{u_n}|^2\text{d}x+\frac{1}{4}\int_{\mathbb R^3}|\nabla u_n|^2\text{d}x.
\end{equation}
Combining \eqref{phi_estimate} and \eqref{eq10} yields
\begin{equation}\label{eq11}
\int_{\mathbb R^3}\phi_{u_n}u_n^2\text{d}x=\int_{\mathbb R^3}|\nabla\phi_{u_n}|^2\text{d}x\geq \int_{\mathbb R^3}|u_n|^3\text{d}x-\frac{1}{4}\int_{\mathbb R^3}|\nabla u_n|^2\text{d}x.
\end{equation}
It follows from \eqref{eq9} and \eqref{eq11} that
\begin{equation*}
	\begin{split}
	&\quad I_{\epsilon}(u_n)-\frac{1}{\mu}\langle I_{\epsilon}'(u_n),u_n\rangle\\
	&\geq \left(\left(\frac{1}{2}-\frac{1}{\mu}\right)-\frac{1}{4}\left(\frac{1}{\mu}-\frac{1}{4}\right)-\delta\left(1+\frac{1}{\mu}\right)\right)\int_{\mathbb R^3}\left(|\nabla u_n|^2+V(x)u_n^2\right)\text{d}x\\
	&\quad\quad+\left(\frac{1}{\mu}-\frac{1}{4}-C_\delta\epsilon\right)\int_{\mathbb R^3}|u_n|^3\text{d}x.
\end{split}\end{equation*}
Let $\mu=3$ and $\delta=\frac{1}{64}$. Then, there exists a constant $\epsilon_*>0$ small enough such that when   $0<\epsilon\leq\epsilon_*$, one has
$$\left(\frac{1}{3}-\frac{1}{4}-C_{\frac{1}{64}}\epsilon\right)\geq0.$$
Consequently,
\begin{equation*}
	I_{\epsilon}(u_n)-\frac{1}{3}\langle I_{\epsilon}'(u_n),u_n\rangle\geq \frac{1}{8}\|u_n\|^{2}_{H^1_r(\mathbb R^3)}.
	\end{equation*}
	
On the other hand, since $\{u_n\}_{n=1}^\infty$ is a (PS) sequence for $I_\epsilon$, it implies
$$c+o_{n}(1)+o_{n}(1)\|u_n\|_{H^1_r(\mathbb R^3)}\geq I_{\epsilon}(u_n)-\frac{1}{3}\langle I_{\epsilon}'(u_n),u_n\rangle.$$
Combining the above two estimates yields
\begin{equation*}
\|u_n\|_{H^1_r(\mathbb R^3)}<\infty.
\end{equation*}
Then there is a subsequence of $\{u_n\}_{n=1}^\infty$ (still labeled by $u_n$) and $u\in H^{1}_{r}(\mathbb{R}^{3})$ such that
\begin{equation}\label{eq12}
	u_n\rightarrow u \text{ weakly in }H^1_r(\mathbb R^3),\ \ u_n\rightarrow u\text{ strongly in }L^p(\mathbb R^3)(2<p<6) \ \ \text{as } n\to\infty.
\end{equation}
The straightforward computations give
\begin{equation}\label{eq13}
	\begin{split}
	&\quad\langle I_{\epsilon}'(u_n)-I_{\epsilon}'(u),u_n-u \rangle\\
	&=\|u_n-u\|^{2}_{H^1_r(\mathbb R^3)}-\int_{\mathbb R^3}(\phi_{u_n}u_n-\phi_uu)(u_n-u)\text{d}x\\
	&\quad\quad-\frac{1}{\epsilon}\int_{\mathbb R^3}(\tilde F'(\epsilon u_n)-\tilde F'(\epsilon u))(u_n-u)\text{d}x.
\end{split}\end{equation}
In terms of \eqref{eq8}, one has
\begin{equation}\label{eq14}
	\frac{1}{\epsilon}\int_{\mathbb R^3}\tilde F'(\epsilon u)(u_n-u)\text{d}x\rightarrow0 \ \ \text{ as }n\rightarrow \infty.
\end{equation}
Furthermore, due to \eqref{eq6}, one has
\begin{equation*}
	\begin{split}
	&\quad\frac{1}{\epsilon}\int_{\mathbb R^3}\left|\tilde F'(\epsilon u_n)(u_n-u)\right|\text{d}x\\
&	\leq\int_{\mathbb R^3}\left(\delta |u_n||u_n-u|+ \epsilon C_\delta u_n^2|u_n-u|\right)\text{d}x\\
&\leq \delta\|u_n\|_{L^2(\mathbb R^3)}\|u_n-u\|_{L^2(\mathbb R^3)}+\epsilon C_\delta\left(\int_{\mathbb R^3}|u_n|^3\text{d}x\right)^{\frac{2}{3}}\left(\int_{\mathbb R^3}|u_n-u|^3\text{d}x\right)^{\frac{1}{3}}.
\end{split}\end{equation*}
This combining with \eqref{eq14} gives that
\begin{equation}\label{eq15}
	\left|\frac{1}{\epsilon}\int_{\mathbb R^3}(\tilde F'(\epsilon u_n)-\tilde F'(\epsilon u))(u_n-u)\text{d}x\right|\leq C\delta\ \ \text{as}~n\rightarrow\infty.
\end{equation}
It follows from \cite[Lemma 1]{MR3148130} that
\begin{equation*}
	\begin{split}
&\quad\left|\int_{\mathbb R^3}(\phi_{u_n}u_n-\phi_uu)(u_n-u)\text{d}x\right|\\
&\leq\int_{\mathbb R^3}\left(\left|\phi_{u_n}u_n(u_n-u)\right|+\left|\phi_uu(u_n-u)\right|\right)\text{d}x\\
&\leq \|\phi_{u_n}\|_{L^6{(\mathbb R^3)}}\|{u_n}\|_{L^{12/5}{(\mathbb R^3)}}\|{u_n-u}\|_{L^{12/5}{(\mathbb R^3)}}+\|\phi_{u}\|_{L^6{(\mathbb R^3)}}\|{u}\|_{L^{12/5}{(\mathbb R^3)}}\|{u_n-u}\|_{L^{12/5}{(\mathbb R^3)}}.
\end{split}\end{equation*}
Note that $\phi_{u_n}\rightarrow\phi_u$ strongly in $L^6(\mathbb R^3)$ and $u_n\rightarrow u$ strongly in $L^{12/5}{(\mathbb R^3)}$. Then
\begin{equation}\label{eq16}
	\left|\int_{\mathbb R^3}(\phi_{u_n}u_n-\phi_uu)(u_n-u)\text{d}x\right| \rightarrow 0\ \ \text{ as }n\rightarrow0.
	\end{equation}
Hence, substituting \eqref{eq15} and \eqref{eq16} into \eqref{eq13} yields
\begin{equation*}
	\|u_n-u\|^{2}_{H^1_r(\mathbb R^3)}\leq o_n(1)+C\delta.
\end{equation*}
Since  $\delta$ is arbitrary, one has
\begin{equation*}
	\|u_n-u\|_{H^1_r(\mathbb R^3)}=0\ \ \text{as }n\rightarrow\infty.
\end{equation*}
This completes the lemma.
\end{proof}
By Lemma \ref{lem1}, \ref{lem2} and the famous mountain pass theorem, there is a $u_{\epsilon}\in H^1_r(\mathbb R^3)$ such that $I_{\epsilon}(u_{\epsilon})=c_{\epsilon}$ and $I_{\epsilon}'(u_{\epsilon})=0$. Therefore, $u_{\epsilon}$ is the weak solution of system \eqref{eq7}. Indeed, the mountain pass value is defined as follows
$$c_{\epsilon}=\inf\limits_{\gamma\in\Gamma}\sup\limits_{t\in[0,1]}I_{\epsilon}(\gamma(t)),$$
where
$$\Gamma=\{\gamma\in C([0,1],H^{1}_{r}(\mathbb{R}^3)):~\gamma(0)=0,\gamma(1)=\gamma_0(1)~\text{and}~I_{\epsilon}(\gamma_0(1))<0\}.$$
By Lemma \ref{lem1}-(i) we have
\begin{equation}\label{3.10}
c_{\epsilon}\ge \beta>0.
\end{equation}

Since $\rho$ and $\beta$ are positive constants independent of $\epsilon$. And the function $\zeta\in H^{1}_{r}(\mathbb{R}^3)$ in Lemma \ref{lem1}-(ii) is also independent of $\epsilon$. Then we have the following result:
\begin{lemma}
	There exist two positive constants \(m_{1},m_{2}>0\), independent of \(\epsilon\in(0,1]\), such that
	\[
	m_{1}\ \le\ I_{\epsilon}(u_{\epsilon})=c_{\epsilon}\ \le\ m_{2},
	\]
	where \(u_{\epsilon}\in H^{1}_{r}(\mathbb{R}^3)\) is the mountain pass critical point of \(I_{\epsilon}\).
\end{lemma}

\begin{proof}
	By Lemma~\ref{lem1} and the Palais–Smale condition (Lemma~\ref{lem2}), for every \(\epsilon\in(0,1]\) there exists
	\(u_{\epsilon}\in H^{1}_{r}(\mathbb{R}^3)\) with
	\[
	I_{\epsilon}'(u_{\epsilon})=0~\text{and}~I_{\epsilon}(u_{\epsilon})=c_{\epsilon}\ge \beta>0,
	\]
	where \(\beta\) is independent of \(\epsilon\). Thus, we choose \(m_{1}:=\beta\).
	
	For the upper bound, consider the comparison functional (the case \(\epsilon=1\))
\begin{equation*}
			\begin{aligned}
	\text{$I_{\epsilon}(u)$}&=\frac{1}{2}\int_{\mathbb R^3}\left(|\nabla u|^2+V(x)u^2\right)\text{d}x-\frac{1}{4}\int_{\mathbb R^3}\phi_{u} u^2\text{d}x-\frac{1}{\epsilon^2}\int_{\mathbb R^3}\tilde F(\epsilon u)\text{d}x\\
    &\leq \frac{1}{2}\int_{\mathbb R^3}\left(|\nabla u|^2+V(x)u^2\right)\text{d}x-\frac{1}{4}\int_{\mathbb R^3}\phi_{u} u^2\text{d}x+\frac{1}{\epsilon^2}\int_{\mathbb R^3}\left(\epsilon^2\delta |u|^2+\epsilon^3 C_\delta |u|^3\right)\text{d}x\\
    &\leq \frac{1}{2}\int_{\mathbb R^3}\left(|\nabla u|^2+V(x)u^2\right)\text{d}x-\frac{1}{4}\int_{\mathbb R^3}\phi_{u} u^2\text{d}x+\int_{\mathbb R^3}\left(\delta |u|^2+C_\delta |u|^3\right)\text{d}x:=J(u).
	\end{aligned}
		\end{equation*}
	Since \(\epsilon\le 1\), we have the point-wise comparison
	\begin{equation*}
		I_{\epsilon}(u)\ \le\ J(u)\quad\text{for~all}~u\in H^{1}_{r}(\mathbb{R}^3).
	\end{equation*}
	Choose \(h\in C_{0}^{\infty}(\mathbb{R}^{3})\) with \(h\neq0\).
	As in the proof of Lemma~\ref{lem1}-(ii), there exists \(s\gg1\) such that \(J(sh)<0\).
	Fix the path \(\gamma_{0}(t):=t\,sh\) on \(t\in[0,1]\).
	Then, by definition of the mountain pass level,
	\[
	c_{\epsilon}\ =\ \inf_{\gamma\in\Gamma}\;\sup_{t\in[0,1]} I_{\epsilon}(\gamma(t))
	\ \le\ \sup_{t\in[0,1]} I_{\epsilon}(\gamma_{0}(t))
\le\ \max_{t\in[0,1]} J(\gamma_{0}(t))=:m_{2}.
	\]
	The map \(t\mapsto J(\gamma_{0}(t))\) is continuous on the compact interval \([0,1]\); hence the supremum is a finite maximum, and \(m_{2}\) is independent of \(\epsilon\).
	Therefore \(I_{\epsilon}(u_{\epsilon})=c_{\epsilon}\le m_{2}\), which, together with the lower bound, proves the lemma.
\end{proof}

The following lemma enables us to remove the truncation \eqref{eq5}.
\begin{lemma}\label{lem3}
	 There exists $\epsilon'<\epsilon_{\ast}$ such that for $0<\epsilon<\epsilon'$, let $u_{\epsilon}$ be the weak solution of system \eqref{eq7} obtained in Lemma \ref{lem2}. Then there exists a positive constant $m_0$ independent of $\epsilon$ such that
	\begin{equation}
	    \label{u_epsion}
	\|u_{\epsilon}\|_{L^{\infty}(\mathbb R^3)}\leq m_0.\end{equation}
\end{lemma}
\begin{proof}
		For any $0<\epsilon<\epsilon_*$, let $u_\epsilon$ satisfy $I_{\epsilon}(u_{\epsilon})=c_{\epsilon}$ and $I_{\epsilon}'(u_{\epsilon})=0$ where $c_{\epsilon}\leq m_2$. Using the similar method in Lemma \ref{lem2}, it implies
		\begin{equation*}
	I_{\epsilon}(u_\epsilon)=I_{\epsilon}(u_\epsilon)-\frac{1}{3}\langle I_{\epsilon}'(u_\epsilon),u_\epsilon\rangle\geq \frac{1}{8}\|u_\epsilon\|^{2}_{H^1_r(\mathbb R^3)}.
	\end{equation*}
	It is easy to say that $\{u_{\epsilon}\}$ is a bounded sequence. Therefore, there exists $M_{0}> 0$ independent of $\epsilon $ such that for all $\epsilon \in(0,\epsilon_*)$, it holds
		\begin{equation*}
			\left\| u_{\epsilon }\right\|_{H_{r}^{1}\left ( \mathbb{R}^{3} \right )}\leq M_{0}.
		\end{equation*}
	
        For any $\alpha > 1$, define
		\begin{equation*}
			\varphi =u_{\epsilon }\left|u_{\epsilon ,k} \right|^{2\left ( \alpha -1 \right )}~\text{and}~u _{\epsilon,k}=
			\begin{cases}
				u _{\epsilon },~\text{if}~\left|u _{\epsilon } \right|\leq k,\\
				k,\ ~~\text{if}~\left|u _{\epsilon } \right|> k.
			\end{cases}
		\end{equation*}
		Since $u_{\epsilon}$ is a critical point of $I_{\epsilon}\left ( u \right )$, one has $\left<I_{\epsilon}'\left (u_{\epsilon}  \right ) ,\varphi \right>=0$, i.e.,
		\begin{equation} \label{eq17}
			\int_{\mathbb{R}^{3}}\left(\nabla u_{\epsilon }\nabla \varphi+V(x)u _{\epsilon }\varphi\right) \text{d}x=\int_{\mathbb{R}^{3}}\phi _{u _{\epsilon}} u_{\epsilon} \varphi \text{d}x-\frac{1}{\epsilon }\int_{\mathbb{R}^{3}}\widetilde{F}'\left ( \epsilon u_{\epsilon} \right )\varphi \text{d}x.
		\end{equation}
		For simplicity, we  refer to $u_{\epsilon}$ as $u$, $u_{\epsilon ,k}$ as $u_{k} $. Substituting $\varphi =u\left|u_{k} \right|^{2\left ( \alpha -1 \right )}$ into the left side of \eqref{eq17} yields
		\begin{equation*}
			\begin{aligned}
				&\quad \int_{\mathbb{R}^{3}}\left|u_{k} \right|^{2\left ( \alpha -1 \right )} \left|\nabla u\right|^{2}\text{d}x+\int_{\mathbb{R}^{3}}V(x) u^{2}\left| u_{k}\right|^{2\left ( \alpha -1 \right )}\text{d}x+2\left ( \alpha -1 \right )\int_{\mathbb{R}^{3}}uu_{k} \left| u_{k}\right|^{2\alpha -4}\nabla u \nabla u_{k}\text{d}x\\
				&\geq \frac{1}{\alpha ^{2}}\int_{\mathbb{R}^{3}}\left|\nabla w_{k}\right|^{2}\text{d}x+\int_{\mathbb{R}^{3}}V(x) w_{k}^{2}\text{d}x\\
				&\geq \frac{C}{\alpha ^{2}}\left ( \int_{\mathbb{R}^{3}}w_{k}^{6}\text{d}x\right )^{\frac{1}{3}}+V_0\int_{\mathbb{R}^{3}}w_{k}^{2}\text{d}x,\\
			\end{aligned}
		\end{equation*}
		where $w _{k}=u \left|u _{k}\right|^{\alpha -1}$.
		
		Substituting $\varphi =u \left|u_{k} \right|^{2\left ( \alpha -1 \right )}$ into the right side of \eqref{eq17} yields
		\begin{equation*}
			\begin{aligned}
				&\quad\left|\int_{\mathbb{R}^{3}}\phi _{u} u^{2} \left|u_{k} \right|^{2\left ( \alpha -1 \right )}\text{d}x-\frac{1}{\epsilon }\int_{\mathbb{R}^{3}}\widetilde{F}'\left ( \epsilon u \right ) u\left| u_{k}\right|^{2\left ( \alpha -1 \right )}\text{d}x\right|\\
				&\leq \left ( \int_{\mathbb{R}^{3}}\phi _{u}^{6} \text{d}x\right )^{\frac{1}{6}}\left ( \int_{\mathbb{R}^{3}}\left|u^{2} \left|u_{k} \right|^{2\left ( \alpha -1 \right )}\right|^{\frac{6}{5}}\text{d}x\right )^{\frac{5}{6}}+\frac{1}{\epsilon }\int_{\mathbb{R}^{3}}\left|\widetilde{F}'\left (\epsilon u \right )u \right| \left| u_{k}\right|^{2\left ( \alpha -1 \right )}\text{d}x\\
				&\leq C \left\|u \right\|_{H_{r}^{1}(\mathbb{R}^3)}^{2}\left ( \int_{\mathbb{R}^{3}}\left|u^{2}\left|u_{k} \right|^{2\left ( \alpha -1 \right )} \right| ^{\frac{6}{5}}\text{d}x\right )^{\frac{5}{6}}+\frac{1}{\epsilon }\int_{\mathbb{R}^{3}}\left| \delta  \epsilon u+C_{\delta} \epsilon ^{2} u^2\right|\left|u\right|\left|u_{k} \right|^{2\left ( \alpha -1 \right )}\text{d}x\\
				&\leq C M^{2}_{0}\left ( \int_{\mathbb{R}^{3}}\left|u^{2}\left|u_{k} \right|^{2\left ( \alpha -1 \right )} \right| ^{\frac{6}{5}}\text{d}x\right )^{\frac{5}{6}}+\delta \int_{\mathbb{R}^{3}}u^{2}\left|u_{k} \right|^{2\left ( \alpha -1 \right )}\text{d}x+C_{\delta }\int_{\mathbb{R}^{3}}\epsilon \left| u\right|^{3}\left| u_{k}\right|^{2\left ( \alpha -1 \right )}\text{d}x.\\
			\end{aligned}
		\end{equation*}
		Then take $\delta=\frac{V_0}{2}$ and $\epsilon'=\min\{C_{\frac{V_0}{2}}^{-1},\epsilon_{\ast}\}$ for $0<\epsilon\leq\epsilon'$, one has
		\begin{equation*}
			\begin{aligned}
            &\quad\left|\int_{\mathbb{R}^{3}}\phi _{u} u^{2} \left|u_{k} \right|^{2\left ( \alpha -1 \right )}\text{d}x-\frac{1}{\epsilon }\int_{\mathbb{R}^{3}}\widetilde{F}'\left ( \epsilon u \right ) u \left| u_{k}\right|^{2\left ( \alpha -1 \right )}\text{d}x\right|\\				
            &\leq C_{1}\left ( \int_{\mathbb{R}^{3}}\left|w_{k} \right|^{\frac{12}{5}} \text{d}x\right )^{\frac{5}{6}}+\frac{V_0}{2}\int_{\mathbb{R}^{3}}w_{k} ^{2} \text{d}x+\epsilon C_{\frac{V_0}{2}}\int_{\mathbb{R}^{3}}w_{k}^{2} \left|u \right|\text{d}x\\
				&\leq C_{1}\left ( \int_{\mathbb{R}^{3}}\left|w_{k} \right|^{\frac{12}{5}} \text{d}x\right )^{\frac{5}{6}}+\epsilon C_{\frac{V_0}{2}}\left ( \int_{\mathbb{R}^{3}} u^{6}\text{d}x \right )^{\frac{1}{6}} \left ( \int_{\mathbb{R}^{3}}|w_{k}|^{ \frac{12}{5}}\text{d}x\right )^{\frac{5}{6}}+\frac{V_0}{2} \int_{\mathbb{R}^{3}}w_{k}^{2}\text{d}x\\
				&\leq C_{1}\left ( \int_{\mathbb{R}^{3}}\left|w_{k} \right|^{\frac{12}{5}} \text{d}x\right )^{\frac{5}{6}}+\left ( \int_{\mathbb{R}^{3}}\left|w_{k} \right| ^{\frac{12}{5}}\text{d}x\right )^{\frac{5}{6}}+\frac{V_0}{2} \int_{\mathbb{R}^{3}} w_{k}^{2}\text{d}x.\\
			\end{aligned}
		\end{equation*}
		Combining the two inequalities above yields
		\begin{equation*}
			\frac{C}{\alpha ^{2}}\left ( \int_{\mathbb{R}^{3}}w_{k}  ^{6}\text{d}x\right )^{\frac{1}{3}}\leq (C_{1}+1)\left ( \int_{\mathbb{R}^{3}}\left|w_{k} \right|^{\frac{12}{5}} \text{d}x\right )^{\frac{5}{6}}.
		\end{equation*}
		Then one gets $$\frac{1}{\alpha ^{2}} \left\|w_{k} \right\|_{L^{6}(\mathbb{R}^3)}^{2}\leq \left ( C_{1} +C_{2}\right )\left\|w_{k} \right\|_{L^{\frac{12}{5}}(\mathbb{R}^3)}^{2}.$$
		Taking $\mu ^{*}=\frac{12}{5}\in \left ( 2,6 \right )$ gives
		\begin{equation} \label{eq18}
			\left\|w_{k} \right\|_{L^{6}(\mathbb{R}^3)}\leq \alpha  C \left\| w_{k}\right\|_{L^{\mu^{\ast}}(\mathbb{R}^3)}.
		\end{equation}
		If $u^{\alpha }\in L^{\frac{12}{5}}\left ( \mathbb{R}^{3} \right )$ and let $k\to \infty$ in both side of \eqref{eq18}, one has
		\begin{equation} \label{eq19}
			\left\|u \right\|_{L^{6 \alpha}(\mathbb{R}^3) }\leq \alpha ^{\frac{1}{\alpha }} C^{\frac{1}{\alpha }} \left\|u \right\|_{L^{\mu^{\ast} \alpha}(\mathbb{R}^3) }.
		\end{equation}
		Set $\alpha _{n}=\left ( \frac{5}{2} \right )^{n}$, $n=1,2,\cdots$. Clearly, $\alpha _{1}= \frac{5}{2}>1$, $2^{*}\alpha _{n}=\mu ^{*}\alpha _{n+1}$, i.e., $6 \alpha_{n}=\frac{12}{5}\alpha _{n+1}$ and $\alpha _{n}\to \infty$. Iterating  \eqref{eq19} with the sequence $\left\{ \alpha _{n}\right\}$ gives
		\begin{equation} \label{eq20}
			\left\|u \right\|_{L^{6 \alpha _{n}}(\mathbb{R}^3)}\leq \alpha _{1}^{\sum_{i=1}^{n}\frac{i}{\alpha _{i}}}C^{\sum_{i=1}^{n}\frac{1}{\alpha _{i}}} \left\|u \right\|_{L^{6}(\mathbb{R}^3)}~\text{for~all}~n\in \mathbb{N}.
		\end{equation}
		Note that $\sum_{i=1}^{\infty }\frac{i}{\alpha _{i}}=\frac{10}{9}$ and $\sum_{i=1}^{\infty }\frac{1}{\alpha _{i}}=\frac{2}{3}$.
		Taking $n\to \infty $ on both sides of \eqref{eq20} gives
		\begin{equation*}
			\left\|u \right\|_{L^{\infty}(\mathbb{R}^3) }\leq \left ( \frac{5}{2} \right )^{\frac{10}{9}} C^{\frac{2}{3}}\left\|u \right\|_{L^{6}(\mathbb{R}^3)}\leq C' \left\|u \right\|_{H_{r}^{1}(\mathbb{R}^3)}.
		\end{equation*}
		Therefore, we get
		\begin{equation*}
			\left\| u_{\epsilon }\right\|_{L^{\infty}(\mathbb{R}^3) }\leq C' \left\|u_{\epsilon} \right\|_{H_{r}^{1}(\mathbb{R}^3)}\leq C' M_{0}.
		\end{equation*}
		Let $C'M_{0}=m_0$. Hence, $\left\|u_{\epsilon} \right\|_{L^{\infty}(\mathbb{R}^3)}\leq m_0$.
The proof is completed.
	\end{proof}
{\bf The proof of Theorem \ref{thm}.}
According to Lemmas \ref{lem1} and \ref{lem2}, for $0<\epsilon<\epsilon'$, there exists a nontrivial weak solution $u$ satisfies
\begin{equation*}
	\begin{cases}
		-\Delta u+V(x) u-\phi u=\frac{1}{\epsilon}\tilde F'(\epsilon u) &\quad\text{in }\mathbb R^3,\\
		-\Delta\phi= u^2 &\quad\text{in }\mathbb R^3.
	\end{cases}
\end{equation*}
Furthermore, take $\epsilon^{*}=\min\{\frac{1}{m_{0}^{2}},\epsilon'\}$. For $0<\epsilon<\epsilon^{*}$, it follows from \eqref{u_epsion} that
	\begin{equation*}
	\begin{split}
\tilde F'(\epsilon u)&=d'(\epsilon u/\sqrt{\epsilon}) F(\epsilon u)\frac{1}{\sqrt{\epsilon}}+d(\epsilon u/\sqrt{\epsilon})f(\epsilon u)-d'(\epsilon u/\sqrt{\epsilon})\frac{1}{\sqrt{\epsilon}}|t|^3+3(1-d(\epsilon u/\sqrt{\epsilon}))|\epsilon u|\epsilon u\\
&=f(\epsilon u),~a.e.~\text{in}~x\in\mathbb{R}^3,
	\end{split}\end{equation*}
where the fact $|\epsilon u/\sqrt{\epsilon}|\leq m_0\sqrt{\epsilon}\leq 1$ has been used.	

Hence, for $0<\epsilon<\epsilon^{*}$, $u$ is also the solution of system \eqref{eq2}. Consequently, $\epsilon u$ is the solution of system $\eqref{eq1}$ and $\lambda^*=\frac{1}{(\epsilon^*)^{2}}$. Thus Theorem \ref{thm} has been proved.

\section{ Schr\"{o}dinger-Poisson system with the sign-changing potential}

The purpose of this section is to obtain a nontrivial solutions of system (\ref{eq1}) with the sign-changing potential.

Before proving our results, we shall introduce the suitable function space where the critical points of the modified functional $I_{\epsilon,R}$ belong to. Define the working space as follows:
$$H^{1}_{V}(\mathbb{R}^3)=\{u\in H^{1}(\mathbb{R}^3):~\int_{\mathbb{R}^3}V(x)u^{2}\text{d}x<\infty \},$$
where the inner product
$$\langle u,v\rangle=\int_{\mathbb{R}^3}\left(\nabla u\cdot\nabla v+\tilde V(x)u v\right)\text{d}x$$ with $\tilde V$ defined in \eqref{tilde_V}.
Then the corresponding norm is $\|u\|_{H^{1}_{V}(\mathbb{R}^3)}=\langle u,u\rangle^{\frac{1}{2}}$.

In this case, we consider the following modified functional. For fixed $R> 1$, we define a truncation function
\begin{equation*}
	\eta _{R}\left ( t \right )=\left\{\begin{matrix}
    \begin{aligned}
        &1,~\text{if}~\left| t\right|\leq R-1,\\
    &0,~\text{if}~\left| t\right|\geq R,
    \end{aligned}
    \end{matrix}\right.
\end{equation*}
where $\eta _{R}\left ( t \right )\in C_0^{\infty}(\mathbb{R},\left [ 0,1 \right ])~~\text{and}~\left | \eta _{R}'\left ( t \right )\right | \leq 2$. For $\epsilon =\frac{1}{\sqrt{\lambda }}$, we define the modified functional $I_{\epsilon,R}\left ( u\right )$ as follows:
\begin{equation*}
    I_{\epsilon,R }\left ( u \right )=\frac{1}{2}\int_{\mathbb{R}^3}\left(|\nabla u|^{2}+V(x)u^2\right)\text{d}x-\frac{1}{4}\int_{\mathbb{R}^{3}}\phi _{u}u^{2}\text{d}x-\frac{1}{\epsilon ^{2}}\eta _{R}\left (\|u\|^{2}_{H^1_{V}(\mathbb{R}^3)}\right )\int_{\mathbb{R}^{3}}F\left (\epsilon \eta _{R} \left ( \left | u\right |\right )u\right )\text{d}x.
\end{equation*}

It is easy to check that $I_{\epsilon ,R }\in C^{1}\left ( H_{V}^{1}\left ( \mathbb{R}^{3} \right ) ,\mathbb{R}\right )$. And for any $u,\varphi \in H_{V}^{1}\left ( \mathbb{R}^{3} \right )$, we have
\begin{equation*}
    \begin{aligned}
		\left< I'_{\epsilon,R  }\left ( u \right ),\varphi \right>=& \int_{\mathbb{R}^{3}}\left ( \nabla u \nabla \varphi +V(x) u\varphi \right )\text{d}x-\int_{\mathbb{R}^{3}}\phi _{u}u\varphi \text{d}x\\
        &-\frac{2}{\epsilon^{2}}\eta _{R}'\left ( \|u\|^{2}_{H^1_{V}(\mathbb{R}^3)} \right )\int_{\mathbb{R}^{3}}(\nabla u\nabla\varphi+\tilde V(x)u\varphi) \text{d}x\int_{\mathbb{R}^{3}}F\left (\epsilon\eta _{R}\left ( \left | u\right | \right ) u\right )\text{d}x\\
        &-\frac{1}{\epsilon}\eta _{R}\left ( \|u\|^{2}_{H^1_{V}(\mathbb{R}^3)} \right )\int_{\mathbb{R}^{3}}\left(f\left (\epsilon \eta _{R}\left ( \left | u\right | \right ) u\right )\left (\eta _{R}\left ( \left | u\right | \right ) \varphi +\eta _{R}'\left ( \left | u\right |\right )|u|\varphi \right )\right)\text{d}x.
    \end{aligned}
\end{equation*}


If $(V_1)$ holds, 
the embedding $H^{1}_{V}(\mathbb{R}^3)\hookrightarrow L^{s}(\mathbb{R}^3)$ is compact for any $s\in[2,6)$ (see Bartsch and  Wang \cite{MR1349229}). 
Let $-\infty<\lambda_1\leq\lambda_2\leq\cdots\leq\lambda_{n}\leq\cdots$
be the complete sequence of eigenvalues of
\begin{equation*}
-\Delta u +V(x)u=\lambda u,\ u\in H^{1}_{V}(\mathbb{R}^3).
\end{equation*}
Each eigenvalue is repeated according to its multiplicity, and let $e_1,e_2,\cdots$ be the corresponding orthonormal eigenfunctions in $L^{2}(\mathbb{R}^N)$.

\subsection{Cerami condition}
    In order to prove Theorem \ref{thm2}, we firstly need to show that the functional $I_{\epsilon,R}$ satisfies the Cerami conditions. To this end, we rewrite $I_{\epsilon,R}$ in the following form
\begin{equation*}
        \begin{aligned}
        I_{\epsilon,R}(u)&=\frac{1}{2}\int_{\mathbb{R}^3}\left(|\nabla u|^{2}+\tilde V(x)u^{2}\right)\text{d}x-\frac{1}{4}\int_{\mathbb{R}^3}\phi_{u}u^{2}\text{d}x\\
    &\indent-\frac{1}{\epsilon ^{2}}\eta _{R}\left (\|u\|^{2}_{H^1_{V}(\mathbb{R}^3)}\right )\int_{\mathbb{R}^{3}}F\left (\epsilon \eta _{R} \left ( \left | u\right |\right )u\right )\text{d}x-\frac{m}{2}\int_{\mathbb{R}^3}u^{2}\text{d}x.
\end{aligned}
    \end{equation*}
   We also consider the following functional
\begin{equation}\label{I0}
I_{0}(u)=\frac{1}{2}\int_{\mathbb R^3}\left(|\nabla u|^2+V(x)u^2\right)\text{d}x-\frac{1}{4}\int_{\mathbb R^3}\phi_{u} u^2\text{d}x,
\end{equation}
which is the energy functional of the following Choquard equation
\begin{equation}\label{eq5.1}
			-\Delta u+V(x)u=\frac{1}{4\pi}\int_{\mathbb R^3}\frac{u^2(y)}{|x-y|}\text{d}y \ \ ~\text{in}~\mathbb{R}^3.
\end{equation}
According to ${\bf (f)}$, we can derive the following conclusion.
\begin{lemma}\label{lem3.5}
       For any $R> 1$, there exists a constant $C_{R} > 0$ independent of $\epsilon $, such that for all $\left ( \epsilon ,u \right )\in \left [ 0,1 \right ]\times H_{V}^{1}\left (  \mathbb{R}^{3}\right )$, one has
      \begin{equation*}
      \left | I_{\epsilon,R }\left ( u \right )-I_{0}\left ( u \right )\right |\leq \frac{1}{R}+C_{R}\epsilon ^{p-2}
      \end{equation*}
      and
      \begin{equation*}
      \left | \langle I_{\epsilon,R }'\left ( u \right ),u\rangle-\langle I_{0}'\left ( u \right ),u\rangle\right |\leq \frac{1}{R}+C_{R}\epsilon ^{p-2}.
        \end{equation*}
\end{lemma}
\begin{proof}
   By ${\bf (f)}$, for any $ \delta >0$ and $R>1$, there exists a constant $C_{\delta,R}>0$ such that $|F\left (t \right )|\leq \delta  t^{2}+C_{\delta,R}\left | t\right |^{p}$ and $\left | f\left ( t \right )\right |\leq \delta \left | t\right |+C_{\delta ,R}\left | t\right |^{p-1}$ for $\left | t\right |\leq R$ and some $p\in(2,6)$. Then for any $u\in  H_{V}^{1}\left ( \mathbb{R}^{3} \right )$, consider
    \begin{equation*}
        \begin{aligned}
            \left | I_{\epsilon,R }\left ( u \right )-I_{0}\left ( u \right )\right |&=\left |\frac{1}{\epsilon ^{2}} \eta _{R}\left ( \|u\|^{2}_{H^1_{V}(\mathbb{R}^3)}\right ) \int_{\mathbb{R}^{3}}F\left (\epsilon\eta _{R} \left ( \left | u\right | \right )u\right )\text{d}x\right |\\
            &\leq \frac{1}{\epsilon^{2}} \eta _{R}\left (\|u\|^{2}_{H^1_{V}(\mathbb{R}^3)} \right )
            \int_{\mathbb{R}^{3}}\left ( \delta \epsilon ^{2} \eta _{R}^{2}\left ( \left | u\right | \right )u^{2}+C_{\delta,R}\epsilon^{p}\eta _{R}^{p}\left ( \left | u\right | \right )\left | u\right |^{p}\right )\text{d}x\\
            & \leq\delta \eta _{R}\left ( \|u\|^{2}_{H^1_{V}(\mathbb{R}^3)}\right )\int_{\mathbb{R}^{3}}u^{2}\text{d}x +C_{\delta,R}\epsilon^{p-2}\eta _{R}\left ( \|u\|^{2}_{H^1_{V}(\mathbb{R}^3)}\right )\int_{\mathbb{R}^{3}}\left | u\right |^{p}\text{d}x \\
            &\leq \delta R+C_{\delta,R}\epsilon^{p-2}.
        \end{aligned}
    \end{equation*}
    Let $\delta = \frac{1}{R^{2}}$, one has
    \begin{equation*}
         \left | I_{\epsilon ,R }\left ( u \right )-I_{0}\left ( u \right )\right |\leq \frac{1}{R}+C_{R}\epsilon ^{p-2}.
    \end{equation*}
For the second estimate in the lemma, we can obtain it by the same method. Thus the proof is completed.
\end{proof}

\begin{lemma}\label{lem4}
Let $\{u_n\}_{n=1}^\infty\subset H^1_V(\mathbb R^3)$ be a Cerami sequence for the functional $I_{\epsilon,R}$. For any $R>1$ fixed, there exists a constant $\epsilon(R)>0$ such that if $\epsilon\in(0,\epsilon(R))$, then $u_n\rightarrow u\text{ in }H^1_V(\mathbb R^3)$ with $u\in H^1_V(\mathbb R^3)$.
\end{lemma}
\begin{proof}
    Let $\{u_n\}$ be a Cerami sequence of $I_{\epsilon,R}$, that is, for some $c\in \mathbb{R}$,
    $$I_{\epsilon,R}(u_n)\to c~\text{and}~(1+\|u_n\|_{H^{1}_{V}(\mathbb R^3)})I_{\epsilon,R}'(u_n)\to0\ ~\text{as}~n\to\infty.$$

Firstly, we claim that
$$\rho_n:=\|u_n\|_{H^{1}_{V}(\mathbb R^3)}<\infty.$$
If it is not true, we may assume $\rho_n\to\infty$. Let $v_n=\frac{u_n}{\rho_n}$. Consequently, $\{v_n\}$ is bounded in $H^{1}_{V}(\mathbb{R}^3)$. Up to a subsequence, $H^{1}_{V}(\mathbb{R}^{3})$ is compactly embedded in $L^2(\mathbb{R}^3)$, we may assume that
\begin{equation*}
		\begin{split}
	v_n&\to v~\text{weakly in}~H^{1}_{V}(\mathbb{R}^3);\\
	v_n&\to v~\text{strongly in}~L^{2}(\mathbb{R}^3);\\
	v_n&\to v~\text{a.e. in}~ \mathbb{R}^3.
	\end{split}\end{equation*}
Consider
\begin{equation*}
		\begin{split}
		c+o_{n}(1)&=I_{\epsilon,R}(u_n)-\frac{1}{3}\langle I_{\epsilon,R}'(u_n),u_n\rangle\\
		&=\left(\frac{1}{2}-\frac{1}{3}\right)\int_{\mathbb R^3}\left(|\nabla u_n|^2+\tilde V(x)u_n^2\right)\text{d}x+\left(\frac{1}{3}-\frac{1}{4}\right)\int_{\mathbb R^3}\phi_{u_n}u_n^2\text{d}x\\
        &\quad\quad-\frac{1}{\epsilon^2}\eta _{R}\left (\|u_n\|^{2}_{H^1_{V}(\mathbb{R}^3)}\right )\int_{\mathbb{R}^{3}}F\left (\epsilon \eta _{R} \left ( \left | u_n\right |\right )u_n\right )\text{d}x-\left(\frac{1}{2}-\frac{1}{3}\right)\int_{\mathbb{R}^3}m u^{2}_n\text{d}x\\
		&\quad\quad-\frac{2}{3\epsilon^{2}}\eta _{R}'\left ( \|u_n\|^{2}_{H^1_{V}(\mathbb{R}^3)} \right )\int_{\mathbb{R}^{3}}(|\nabla u_n|^2+\tilde V(x)u_n^2) \text{d}x\int_{\mathbb{R}^{3}}F\left (\epsilon\eta _{R}\left ( \left | u_n\right | \right ) u_n\right )\text{d}x\\
        &\quad\quad-\frac{1}{3\epsilon}\eta _{R}\left ( \|u_n\|^{2}_{H^1_{V}(\mathbb{R}^3)} \right )\int_{\mathbb{R}^{3}}\left(f\left (\epsilon \eta _{R}\left ( \left | u_n\right | \right ) u_n\right )\left (\eta _{R}\left ( \left | u_n\right | \right ) u_n +\eta _{R}'\left ( \left | u_n\right |\right )u_n^2 \right )\right)\text{d}x\\
		&\geq\left(\frac{1}{2}-\frac{1}{3}\right)\int_{\mathbb R^3}\left(|\nabla u_n|^2+\tilde V(x)u_n^2\right)\text{d}x+\left(\frac{1}{3}-\frac{1}{4}\right)\int_{\mathbb R^3}\phi_{u_n}u_n^2\text{d}x\\
		&\quad\quad-\frac{4}{3}\left(\frac{1}{R}+C_{R}\epsilon ^{p-2}\right)-\left(\frac{1}{2}-\frac{1}{3}\right)\int_{\mathbb{R}^3}m u^{2}_n\text{d}x\\
		&=\frac{1}{6}\int_{\mathbb R^3}\left(|\nabla u_n|^2+\tilde V(x)u_n^2\right)\text{d}x-\frac{1}{6}\int_{\mathbb{R}^3}m u^{2}_n\text{d}x-\frac{4}{3}\left(\frac{1}{R}+C_{R}\epsilon ^{p-2}\right).
	\end{split}\end{equation*}
Then, there exists a constant $\epsilon'(R)>0$ small enough such that when   $0<\epsilon\leq\epsilon'(R)$, one has
$$\frac{1}{R}+C_{R}\epsilon ^{p-2}\leq1 $$
and then
\begin{equation*}
	c+o_{n}(1)\geq \frac{1}{6}\|u_n\|^{2}_{H^{1}_{V}(\mathbb R^3)}- \frac{1}{6}\int_{\mathbb{R}^3}m u^{2}_n\text{d}x.
	\end{equation*}
Multiplying both sides by $\rho_{n}^{-2}$, for large $n$, we have 	
	$$\|v_n\|_{L^{2}(\mathbb{R}^3)}^{2}\geq \frac{1}{2m}.$$
Note that $v_n\to v$ strongly in $L^{2}(\mathbb{R}^3)$. Then $v\neq0$. Thus the Lebesgue measure of the set
$$\mathcal{M}=\{x\in\mathbb{R}^3:v(x)\neq 0\}$$
is positive. For $x\in\mathcal{M}$, as $n\to\infty$, we deduce
$$v_n(x)\to v(x)\neq 0,~u_{n}(x)=v_{n}(x)\rho_n\to\infty~\text{and}~\phi_{u_n}(x)\to\infty.$$

On the other hand, it follows from assumptions $({\bf f})$ and $({\bf f'})$ that for $\epsilon\in (0,\frac{\delta_0}{R})$,
\begin{equation}\label{eq21}
   F\left (\epsilon \eta _{R} \left ( \left | t\right |\right )t\right )\geq 0\ ~\text{for all}~t\in\mathbb{R}.
\end{equation}
From \eqref{eq21} and the Fatou lemma, for $\epsilon\in (0,\frac{\delta_0}{R})$ we have
\begin{equation*}
	\begin{split}
&\indent\int_{\mathbb{R}^3}\left(\frac{\phi_{u_n}u^{2}_{n}}{4}+\frac{1}{\epsilon^2}\eta _{R}\left (\|u_n\|^{2}_{H^1_{V}(\mathbb{R}^3)}\right )F\left (\epsilon \eta _{R} \left ( \left | u_n\right |\right )u_n\right )+\frac{1}{2}mu_{n}^{2}\right)\rho^{-2}_n\text{d}x\\
&\geq\int_{\mathbb{R}^3}\frac{\phi_{u_n}u^{2}_{n}}{4u_{n}^{2}}v^{2}_{n}\text{d}x\\
&\geq \frac{1}{4}\int_{\mathcal{M}}\phi_{u_n}v^{2}_n\text{d}x\to\infty,~~\text{as}~n\to\infty.
\end{split}
\end{equation*}
Thus
\begin{equation*}
		\begin{split}
		c-1&\leq I_{\epsilon,R}(u_n)\\
		&=\frac{1}{2}\int_{\mathbb{R}^3}\left(|\nabla u_n|^{2}+\tilde V(x)u_n^{2}\right)\text{d}x-\frac{1}{4}\int_{\mathbb{R}^3}\phi_{u_n}u_n^{2}\text{d}x\\
		&\indent-\frac{1}{\epsilon^{2}}\eta _{R}\left (\|u_n\|^{2}_{H^1_{V}(\mathbb{R}^3)}\right )\int_{\mathbb{R}^3}F\left (\epsilon \eta _{R} \left ( \left | u_n\right |\right )u_n\right )\text{d}x-\frac{m}{2}\int_{\mathbb{R}^3}u_n^{2}\text{d}x\\
		&=\rho_{n}^{2}\left(\frac{1}{2}-\int_{\mathbb{R}^3}\left(\frac{\phi_{u_n}u^{2}_{n}}{4}+\frac{1}{\epsilon^2}\eta _{R}\left (\|u_n\|^{2}_{H^1_{V}(\mathbb{R}^3)}\right )F\left (\epsilon \eta _{R} \left ( \left | u_n\right |\right )u_n\right )+\frac{1}{2}mu_{n}^{2}\right)\rho^{-2}_n\text{d}x\right)\\
        &\to-\infty,~~\text{as}~n\to\infty.
	\end{split}
    \end{equation*}
This is impossible. Therefore, our claim is true.

From the boundedness of $\{u_n\}$ and $H^{1}_{V}(\mathbb{R}^{3})$ is compactly embedded in $L^q(\mathbb{R}^3)(q\in[2,6))$, up to a sub-sequence we may assume
\begin{equation*}
		\begin{split}
	u_n&\to u_0~\text{weakly in}~H^{1}_{V}(\mathbb{R}^3);\\
	u_n&\to u_0~\text{strongly in}~L^{q}(\mathbb{R}^3)(q\in[2,6));\\
	u_n&\to u_0~\text{a.e. in}~ \mathbb{R}^3.
	\end{split}\end{equation*}
Next we prove the strong convergence of $u_{n}$ in $H_{V}^{1}\left ( \mathbb{R}^{3} \right ) $. To simplify the notation, let $$ \Psi_{\epsilon,R}\left ( u \right ):=\frac{1}{\epsilon ^{2}}\eta _{R}\left (\|u\|_{H^1_V(\mathbb{R}^3)}^{2}\right )\int_{\mathbb{R}^{3}}F\left (\epsilon \eta _{R} \left ( \left | u\right |\right )u\right )\text{d}x.$$

{\bf Claim}:
\begin{equation*}
    \begin{aligned}
     &\langle I'_{\epsilon,R}(u_n) - I'_{\epsilon,R}(u_0), u_n - u_0 \rangle \\
&= \|u_n - u_0\|_{H^1_V(\mathbb{R}^3)}^2-\int_{\mathbb{R}^3} (\phi_{u_n}u_n - \phi_{u_0}u_0)(u_n - u_0) \text{d}x -m\int_{\mathbb{R}^3}(u_n-u_0)^2\text{d}x\\
&\quad- \langle \Psi'_{\epsilon,R}(u_n) - \Psi'_{\epsilon,R}(u_0), u_n - u_0 \rangle\\
&\geq \frac{1}{2}\|u_n-u_0\|_{H^{1}_V(\mathbb{R}^3)}^{2}+o_n(1).
    \end{aligned}
\end{equation*}
Indeed, by the boundedness of ${u_n}$ in $H^{1}_{V}(\mathbb{R}^3)$, we have
\begin{equation*}
    \left |\int_{\mathbb{R}^{3}}\left ( \phi _{u_{n}} u_{n}-\phi _{u_{0}}u_{0}\right )\left ( u_{n}-u_{0} \right ) \text{d}x\right |\leq o_{n}\left ( 1 \right )\left\| u_{n}-u_{0}\right\|_{H^{1}_{V}(\mathbb{R}^3)}=o_{n}\left ( 1 \right )
\end{equation*}
and
\begin{equation*}
    \int_{\mathbb{R}^3}(u_n-u_0)^2\text{d}x=o_{n}(1).
\end{equation*}
Thus, we only need to prove that
\begin{equation*}
    \langle \Psi_{\epsilon,R}'(u_n) - \Psi_{\epsilon,R}'(u_0), u_n - u_0 \rangle \leq\frac{1}{2}\|u_n-u_0\|_{H^{1}_V(\mathbb{R}^3)}^{2}+o_n(1).
\end{equation*}
A direct computation shows that
\begin{align*}
\langle \Psi'_{\epsilon,R}(u), w \rangle =
& \frac{2}{\epsilon^2} \eta_R'(\|u\|_{H^1_{V}(\mathbb{R}^3)}^2) \int_{\mathbb{R}^3}(\nabla u\nabla w+\tilde V(x)uw) \text{d}x \int_{\mathbb{R}^3} F(\epsilon \eta_R(|u|) u)  \text{d}x \\
&+ \frac{1}{\epsilon} \eta_R(\|u\|_{H^1_{V}(\mathbb{R}^3)}^2) \int_{\mathbb{R}^3} f(\epsilon \eta_R(|u|) u) \left(\eta_R(|u|) + \eta_R'(|u|)|u| \right) w \text{d}x.
\end{align*}
And then
\begin{align}
      &\left<  \Psi'_{\epsilon,R} \left ( u_{n} \right )-\Psi_{\epsilon ,R}'\left ( u_{0} \right ),u_{n}-u_{0}\right>\nonumber \\
      &= \frac{2}{\epsilon^2} \eta_R'(\|u_n\|_{H^1_{V}(\mathbb{R}^3)}^2) \int_{\mathbb{R}^3}(\nabla u_n\nabla (u_{n}-u_{0})+\tilde V(x)u_n(u_{n}-u_{0})) \text{d}x \int_{\mathbb{R}^3} F(\epsilon \eta_R(|u_n|) u_n) \text{d}x \tag{T1}\\
     &\quad~- \frac{2}{\epsilon^2} \eta_R'(\|u_0\|_{H^1_{V}(\mathbb{R}^3)}^2) \int_{\mathbb{R}^3}(\nabla u_0\nabla (u_{n}-u_{0})+\tilde V(x)u_0(u_{n}-u_{0})) \text{d}x \int_{\mathbb{R}^3} F(\epsilon \eta_R(|u_0|) u_0)  \text{d}x \tag{T2} \\
    &\quad~+ \frac{1}{\epsilon} \eta_R(\|u_n\|_{H^1_{V}(\mathbb{R}^3)}^2) \int_{\mathbb{R}^3} f(\epsilon\eta_R(|u_n|) u_n) \left(\eta_R(|u_n|) + \eta_R'(|u_n|)|u_n| \right) (u_n-u_0) \text{d}x  \tag{T3}\\
    &\quad~- \frac{1}{\epsilon} \eta_R(\|u_0\|_{H^1_{V}(\mathbb{R}^3)}^2) \int_{\mathbb{R}^3} f(\epsilon \eta_R(|u_0|) u_0) \left(\eta_R(|u_0|) + \eta_R'(|u_0|)|u_0| \right) (u_n-u_0) \text{d}x. \tag{T4}
\end{align}
We now proceed to provide term-by-term estimates for $(\text{T}1)$ through $(\text{T}4)$. We begin by estimating $(\text{T}1)$, for which we obtain the following estimate:
\begin{equation*}
    \begin{aligned}
       \left | (\text{T}1)\right |&=\frac{2}{\epsilon^2} \eta_R'(\|u_n\|_{H^1_{V}(\mathbb{R}^3)}^2) \int_{\mathbb{R}^3}(\nabla u_n\nabla (u_{n}-u_{0})+\tilde V(x)u_n(u_{n}-u_{0})) \text{d}x \int_{\mathbb{R}^3} F(\epsilon \eta_R(|u_n|) u_n) \text{d}x\\
       &=\frac{2}{\epsilon^2} \eta_R'(\|u_n\|_{H^1_{V}(\mathbb{R}^3)}^2) \int_{\mathbb{R}^3}|\nabla (u_n-u_0)|^2+\tilde V(x)(u_{n}-u_{0})^2) \text{d}x \int_{\mathbb{R}^3} F(\epsilon \eta_R(|u_n|) u_n) \text{d}x\\
       &\quad+\frac{2}{\epsilon^2} \eta_R'(\|u_n\|_{H^1_{V}(\mathbb{R}^3)}^2) \int_{\mathbb{R}^3}(\nabla u_0\nabla (u_{n}-u_{0})+\tilde V(x)u_0(u_{n}-u_{0})) \text{d}x \int_{\mathbb{R}^3} F(\epsilon \eta_R(|u_n|) u_n) \text{d}x\\
       &\leq  \frac{2}{\epsilon^2} \eta_R'(\|u_n\|_{H^1_{V}(\mathbb{R}^3)}^2)\|u_n-u_0\|^2_{H^{1}_{V}(\mathbb{R}^3)}\int_{\mathbb{R}^3} F(\epsilon \eta_R(|u_n|) u_n)  \text{d}x\\
       &\quad+o_n(1)\frac{2}{\epsilon^2} \eta_R'(\|u_n\|_{H^1_{V}(\mathbb{R}^3)}^2)\int_{\mathbb{R}^3} F(\epsilon \eta_R(|u_n|) u_n) \text{d}x\\
      & \leq \frac{2}{\epsilon^2} \eta_R'(\|u_n\|_{H^1_{V}(\mathbb{R}^3)}^2) \|u_n-u_0\|^2_{H^{1}_{V}(\mathbb{R}^3)}\int_{\mathbb{R}^{3}}\left ( \delta \epsilon ^{2} \eta _{R}^{2}\left ( \left | u_n\right | \right )u_n^{2}+C_{\delta,R}\epsilon ^{p}\eta _{R}^{p}\left ( \left | u_n\right | \right )\left | u_n\right |^{p}\right )\text{d}x\\
      &\quad+o_{n}(1)\frac{2}{\epsilon^2} \eta_R'(\|u_n\|_{H^1_{V}(\mathbb{R}^3)}^2)\int_{\mathbb{R}^{3}}\left ( \delta \epsilon ^{2} \eta _{R}^{2}\left ( \left | u_n\right | \right )u_n^{2}+C_{\delta,R}\epsilon ^{p}\eta _{R}^{p}\left ( \left | u_n\right | \right )\left | u_n\right |^{p}\right )\text{d}x\\
      &\leq\delta \eta_R'(\|u_n\|_{H^1_{V}(\mathbb{R}^3)}^2) \|u_n-u_0\|^2_{H^{1}_{V}(\mathbb{R}^3)}\int_{\mathbb{R}^{3}}u_n^{2}\text{d}x \\
      &\quad+C_{\delta,R}\epsilon ^{p-2}\eta_R'(\|u_n\|_{H^1_{V}(\mathbb{R}^3)}^2)\|u_n-u_0\|^2_{H^{1}_{V}(\mathbb{R}^3)}\int_{\mathbb{R}^{3}}\left | u_n\right |^{p}\text{d}x\\
      &\quad+(\delta \eta_R'(\|u_n\|_{H^1_{V}(\mathbb{R}^3)}^2)\int_{\mathbb{R}^{3}}u_n^{2}\text{d}x+C_{\delta,R}\epsilon ^{p-2}\eta_R'(\|u_n\|_{H^1_{V}(\mathbb{R}^3)}^2)\int_{\mathbb{R}^{3}}\left | u_n\right |^{p}\text{d}x)o_{n}(1) \\
    &\leq \delta R^2\|u_n-u_0\|_{H^{1}_{V}(\mathbb{R}^3)}^{2}+C_{\delta,R}\epsilon ^{p-2}\|u_n-u_0\|_{H^{1}_{V}(\mathbb{R}^3)}^{2}+(\delta R^2+C_{\delta,R}\epsilon ^{p-2})o_{n}(1),
    \end{aligned}
\end{equation*}
where we use the fact that $u_n\to u_0$ weakly in $H^1_{V}(\mathbb{R}^3)$. Take $\epsilon(R)=\min\{\epsilon'(R),\frac{\delta_0}{R},\left(\frac{1}{4C_{\delta,R}}\right)^{1/(p-2)}\}$ with $\delta=\frac{1}{4R^2}$, for any $\epsilon\in(0,\epsilon(R))$ we can deduce that $$|(\text{T}1)|\leq \frac{1}{2}\|u_n-u_0\|_{H^{1}_{V}(\mathbb{R}^3)}^{2}+o_{n}(1).$$
Next, by using $u_{n} \to u _{0}~ \text{weakly~in}~H_{V}^{1}\left ( \mathbb{R}^{3}\right )$ and the similar method in Lemma \ref{lem3.5}, we obtain that
$$|(\text{T}2)|=o_{n}\left ( 1 \right ).$$

We now proceed to $(\text{T}3)$ and $(\text{T}4)$. Without loss of generality, we assume that $\epsilon(R)\in(0,1)$, for any $\epsilon\in(0,\epsilon(R))$, one has
\begin{equation*}
    \begin{aligned}
         \left |(\text{T}3)\right |&\leq   \frac{C_{R}}{\epsilon} \eta_R(\|u_n\|_{H^1_V(\mathbb{R}^3)}^2) \left|\int_{\mathbb{R}^3} f(\epsilon \eta_R(|u_n|) u_n) (u_n - u_0) \text{d}x\right|\\
    &\leq   C_{R}\eta_R(\|u_n\|_{H^1_V(\mathbb{R}^3)}^2) \int_{\mathbb{R}^3} \left [ \delta \eta _{R}\left ( \left | u_{n}\right | \right )|u_{n}|+C_{\delta,R } \epsilon^{p-2}\eta _{R}^{p-1}\left ( \left | u_{n}\right |\right )|u_{n}| ^{p-1}\right ]|u_n - u_0|\text{d}x\\
    &\leq\delta C_{R}\eta_R(\|u_n\|_{H^1_V(\mathbb{R}^3)}^2)   \int_{\mathbb{R}^3}|u_{n}||u_{n}-u_{0} |\text{d}x\\
    &\quad+ C_{\delta ,R}\epsilon^{p-2}\eta_R(\|u_n\|_{H^1_V(\mathbb{R}^3)}^2) \int_{\mathbb{R}^3}|u_{n}|^{p-1}|u_{n}-u_{0}|\text{d}x.\\
    &\leq \delta C_{R}\eta_R(\|u_n\|_{H^1_V(\mathbb{R}^3)}^2) \|u_n\|_{L^{2}(\mathbb{R}^3)}\|u_n-u_0\|_{L^{2}(\mathbb{R}^3)}\\
    &\quad+C_{\delta ,R}\epsilon^{p-2}\eta_R(\|u_n\|_{H^1_V(\mathbb{R}^3)}^2) \left ( \int_{\mathbb{R}^3}|u_{n}|^{p} \text{d}x\right )^{\frac{p-1}{p}}\left ( \int_{\mathbb{R}^3}\left | u_{n}-u_{0}\right |^{p}\text{d}x \right )^{\frac{1}{p}}\\
    &\leq \delta C_{R}\|u_n-u_0\|_{L^{2}(\mathbb{R}^3)}+C_{\delta ,R}\|u_n-u_0\|_{L^{p}(\mathbb{R}^3)}.
    \end{aligned}
\end{equation*}
Then it follows from $u_n\to u_0$ strongly in $L^{q}(\mathbb{R}^3)$ with $q\in[2,6)$ that $|(\text{T}3)|=o_{n}\left ( 1 \right ).$ Finally, it is easy to see that $|(\text{T}4)|=o_{n}\left ( 1 \right ).$

Combining these results, we have proved the above claim, and then one has
\begin{equation*}
    \|u_n - u_0\|_{H^1_V(\mathbb{R}^3)}^2=o_{n}\left ( 1 \right ),
\end{equation*}
which implies $u_{n} \to u _{0}$ strongly in $H_{V}^{1}\left ( \mathbb{R}^{3} \right )$.
\end{proof}
By Morse theory, we know that if $I$ satisfies the Cerami conditions, $u={\bf 0}$ is an isolated critical point of $I$ and the critical points set $\mathcal{K}:=\{u\in H^{1}_{V}(\mathbb{R}^3):I'(u)=0\}=\{\bf 0\}$, then $C_{q}(I,\infty)\cong C_{q}(I,0)$ for all $q\in \mathbb{N}$. If $C_{q}(I,\infty)\ncong C_{q}(I,0)$ for some $q\in \mathbb{N}$ then $I$ must have a nontrivial critical point. Therefore, one has to compute these groups to get the nontrivial critical point.
\begin{proposition}\label{pro1}(see  Bartsch and  Li \cite{MR1420790})
Suppose that $I\in C^{1}(H^{1}_{V}(\mathbb{R}^3),\mathbb{R})$ satisfies the Cerami condition and $C_{k}(I,0)\neq C_{k}(I,\infty)$ for some $k\in\mathbb{N}$, then $I$ has a nontrivial critical point.
\end{proposition}

To compute the critical groups, we need to decompose the workspace $H^{1}_{V}(\mathbb{R}^3)$. Since $0$ is not an eigenvalue of $-\Delta+V(x)$, we can assume that there exists an integer $d\geq0$ such that $0\in(\lambda_{d},\lambda_{d+1})$, where we set $\lambda_{0}=-\infty$. For $d\geq1$, let $H^{1}_{V}(\mathbb{R}^3)^{-}=\mbox{span}\{e_1,\cdots,e_d\}\ \mbox{and}\  H^{1}_{V}(\mathbb{R}^3)^{+}=(H^{1}_{V}(\mathbb{R}^3)^{-})^{\perp}.$  Then $H^{1}_{V}(\mathbb{R}^3)^{-}$ and $H^{1}_{V}(\mathbb{R}^3)^{+}$ are the negative space and positive space of the quadratic form $ B(u)=\frac{1}{2}\int_{\mathbb{R}^3}\left(|\nabla u|^{2}+V(x)u^{2}\right)\text{d}x,$ respectively. Thus, there exits a constant $\eta>0$, such that
\begin{equation}\label{eq22}
\pm B(u)\geq\eta\|u\|_{H^{1}_{V}(\mathbb{R}^3)}^{2},\ u\in  H^{1}_{V}(\mathbb{R}^3)^{\pm}.
\end{equation}

\subsection{Critical groups at zero}
In this section, we will use the following proposition to compute the critical groups of $I_{\epsilon,R}$ at zero.
\begin{proposition}\label{pro2}(see Liu \cite{MR994751})
Suppose $I\in C^{1}(H^{1}_{V}(\mathbb{R}^3),\mathbb{R})$ has a local linking at zero with respect to the decomposition $H^{1}_{V}(\mathbb{R}^3)=H^{1}_{V}(\mathbb{R}^3)^{-}\oplus H^{1}_{V}(\mathbb{R}^3)^{+}$, i.e. for some $r>0$.
\begin{equation*}
\begin{aligned}
I(u)\leq0\  \mbox{for}\ u\in H^{1}_{V}(\mathbb{R}^3)^{-}\cap B_{r},\ \mbox{and}\ I(u)>0\  \mbox{for}\ u\in (H^{1}_{V}(\mathbb{R}^3)^{+}\setminus\{{\bf 0}\})\cap B_{r},
\end{aligned}
\end{equation*}
where $B_{r}=\{u\in H^{1}_{V}(\mathbb{R}^3): \|u\|_{H^{1}_{V}(\mathbb{R}^3)}<r\}$. If $d=\dim H^{1}_{V}(\mathbb{R}^3)^{-}<\infty$, then $C_{d}(I,0)\neq0$.
\end{proposition}
\begin{lemma}\label{lem5}
Under assumptions ${\bf (V_1)}$, ${\bf (V_2)}$, ${\bf (f)}$ and ${\bf (f')}$. For any $R>1$ fixed, there exists $\epsilon''(R)>0$ such that for any $\epsilon\in(0,\epsilon''(R))$ then the functional $I_{\epsilon,R}$ has a local linking at zero with respect to decomposition $H^{1}_{V}(\mathbb{R}^3)=H^{1}_{V}(\mathbb{R}^3)^{-}\oplus H^{1}_{V}(\mathbb{R}^3)^{+}$.
\end{lemma}
\begin{proof}
 Follows from the proof of Lemma \ref{lem3.5}. There exists $\epsilon''(R)>0$, such that for any $\epsilon\in(0,\epsilon''(R))$ and $u\in H^{1}_{V}(\mathbb{R}^3)^{+}\setminus\{\bf 0\}$ with $\|u\|_{H^{1}_{V}(\mathbb{R}^3)}=r$, taking $r>0$ small enough, by \eqref {eq22} we obtain
\begin{equation*}
\begin{aligned}
I_{\epsilon,R}(u)&=\frac{1}{2}\int_{\mathbb R^3}\left(|\nabla u|^2+V(x)u^2\right)\text{d}x-\frac{1}{4}\int_{\mathbb R^3}\phi_{u} u^2\text{d}x-\frac{1}{\epsilon^2}\eta _{R}\left (\|u\|^{2}_{H^1_{V}(\mathbb{R}^3)}\right )\int_{\mathbb{R}^{3}}F\left (\epsilon \eta _{R} \left ( \left | u\right |\right )u\right )\text{d}x\\
&\geq\eta\|u\|^2_{H^{1}_{V}(\mathbb{R}^3)}-\frac{C}{4}\|u\|^4_{H^{1}_{V}(\mathbb{R}^3)}\\
&\quad-\left(\delta\eta _{R}\left ( \|u\|^{2}_{H^1_{V}(\mathbb{R}^3)}\right )\int_{\mathbb{R}^{3}}u^{2}\text{d}x +C_{\delta,R}\epsilon^{p-2}\eta _{R}\left ( \|u\|^{2}_{H^1_{V}(\mathbb{R}^3)}\right )\int_{\mathbb{R}^{3}}\left | u\right |^{p}\text{d}x\right)\\
&\geq\eta\|u\|^2_{H^{1}_{V}(\mathbb{R}^3)}-\frac{C}{4}\|u\|^4_{H^{1}_{V}(\mathbb{R}^3)}-\left(\frac{\eta}{4} \|u\|^2_{H^{1}_{V}(\mathbb{R}^3)} +C\|u\|^p_{H^{1}_{V}(\mathbb{R}^3)}\right)\\
&\geq \frac{\eta}{4}\|u\|^2_{H^{1}_{V}(\mathbb{R}^3)}=\frac{\eta r^2}{4}>0.
\end{aligned}
\end{equation*}
Then we obtain
\begin{equation}\label{eq24}
    I_{\epsilon,R}(u)>0 \  \mbox{for}\ u\in (H^{1}_{V}(\mathbb{R}^3)^{+}\setminus\{{\bf0}\})\cap B_{r}.
\end{equation}

On the other hand, for $u\in H^{1}_{V}(\mathbb{R}^3)^{-}\cap B_{r}$ and $\epsilon\in (0,\frac{\delta_0}{R})$, it is easy to see that
\begin{equation*}
\begin{aligned}
I_{\epsilon,R}(u)&=\frac{1}{2}\int_{\mathbb R^3}\left(|\nabla u|^2+V(x)u^2\right)\text{d}x-\frac{1}{4}\int_{\mathbb R^3}\phi_{u} u^2\text{d}x-\frac{1}{\epsilon^2}\eta _{R}\left (\|u\|^{2}_{H^1_{V}(\mathbb{R}^3)}\right )\int_{\mathbb{R}^{3}}F\left (\epsilon \eta _{R} \left ( \left | u\right |\right )u\right )\text{d}x\\
&\leq-\eta\|u\|^2_{H^{1}_{V}(\mathbb{R}^3)}-\frac{1}{4}\int_{\mathbb R^3}\phi_{u} u^2\text{d}x-\frac{1}{\epsilon^2}\eta _{R}\left (\|u\|^{2}_{H^1_{V}(\mathbb{R}^3)}\right )\int_{\mathbb{R}^{3}}F\left (\epsilon \eta _{R} \left ( \left | u\right |\right )u\right )\text{d}x\\
&\leq-\eta \|u\|^2_{H^{1}_{V}(\mathbb{R}^3)}.
\end{aligned}
\end{equation*}
From this and \eqref{eq24}, it implies $I_{\epsilon,R}$ has a local linking property at zero.
\end{proof}
\subsection{Critical groups at infinity}
\begin{lemma}\label{lem6}
If assumptions ${\bf (V_1)}$, ${\bf (V_2)}$, ${\bf (f)}$ and ${\bf (f')}$ hold. For any $R>1$ fixed, there exists a constant $\epsilon'''(R)>0$ such that for all $\epsilon\in(0,\epsilon'''(R))$, it has $C_{q}(I_{\epsilon,R},\infty)\cong0$ with $q\in \mathbb{N}=\{0,1,2,\cdots\}$.
\end{lemma}
\begin{proof}
Define $S^{\infty}$ as the unit sphere in $H^{1}_{V}(\mathbb{R}^3)$. Firstly, we can show that
\begin{equation}\label{eq25}
I_{\epsilon,R}(sh)\to-\infty \ \mbox{as}\ s\to\infty \ \mbox{for any}\ h\in S^{\infty}.
\end{equation}
Similar with Lemma \ref{lem4}, for any $\epsilon\in (0,\frac{\delta_0}{R})$, from $F\left (\epsilon \eta _{R} \left ( \left | t\right |\right )t\right )\geq 0$ we deduce
\begin{equation*}
		\begin{split}
		I_{\epsilon,R}(sh)&=\frac{1}{2}\int_{\mathbb{R}^3}\left(|\nabla (sh)|^{2}+\tilde V(x)(sh)^{2}\right)\text{d}x-\frac{1}{4}\int_{\mathbb{R}^3}\phi_{(sh)}(sh)^{2}\text{d}x\\
		&\indent-\frac{1}{\epsilon^2}\eta _{R}\left (\|sh\|^{2}_{H^1_{V}(\mathbb{R}^3)}\right )\int_{\mathbb{R}^{3}}F\left (\epsilon \eta _{R} \left ( \left | sh\right |\right )sh\right )\text{d}x-\frac{m}{2}\int_{\mathbb{R}^3}(sh)^{2}\text{d}x\\
		&\leq s^{2}\left(\frac{1}{2}-\frac{s^2}{4}\int_{\mathbb{R}^3}\phi_{h}h^{2}\text{d}x\right)\to-\infty,~~\text{as}~s\to\infty.
	\end{split}\end{equation*}
	
We next shall show the following claim is right.\\
{\bf Claim:} There is a constant $A>0$ such that if $I_{\epsilon,R}(u)\leq-A$, then
$$\frac{\mbox{d}}{\mbox{d}t}|_{t=1}I_{\epsilon,R}(t u)<0.$$
If this claim is false, there exists a sequence $\{u_n\}\subset H^{1}_{V}(\mathbb{R}^3)$ such that
\begin{equation}\label{eq26}
I_{\epsilon,R}(u_n)\leq-n\ \mbox{and}\ \langle I_{\epsilon,R}'(u_n),u_n\rangle=\frac{\mbox{d}}{\mbox{d}t}|_{t=1}I_{\epsilon,R}(t u_n)\geq0.
\end{equation}
Notice that if $\{u_{n}\}_{n=1}^\infty$ is a bounded sequence in $H^{1}_{V}(\mathbb{R}^3)$, then $\{I_{\epsilon,R}(u_{n})\}_{n=1}^{\infty}$ is also bounded. Then by $I_{\epsilon,R}(u_n)\leq -n$, it implies
$$\int_{\mathbb{R}^3}\left(|\nabla u_n|^2+\widetilde{V}(x)u_n^{2}\right)\text{d}x\to\infty \ \mbox{as}\ n\to\infty.$$

For large $n$, as $\|u_n\|_{H^1_{V}}^2>R$ one has 
\begin{equation}\label{eq27}
		\begin{split}
		0&\geq I_{\epsilon,R}(u_n)-\frac{1}{3}\langle I_{\epsilon,R}'(u_n),u_n\rangle\\
		&=\left(\frac{1}{2}-\frac{1}{3}\right)\int_{\mathbb R^3}\left(|\nabla u_n|^2+\tilde V(x)u_n^2\right)\text{d}x+\left(\frac{1}{3}-\frac{1}{4}\right)\int_{\mathbb R^3}\phi_{u_n}u_n^2\text{d}x\\
        &\quad\quad-\left(\frac{1}{2}-\frac{1}{3}\right)\int_{\mathbb{R}^3}m u^{2}_n\text{d}x\\
		&\geq\left(\frac{1}{2}-\frac{1}{3}\right)\int_{\mathbb R^3}\left(|\nabla u_n|^2+\tilde V(x)u_n^2\right)\text{d}x-\left(\frac{1}{2}-\frac{1}{3}\right)\int_{\mathbb{R}^3}m u^{2}_n\text{d}x.
	\end{split}\end{equation}

Setting $v_{n}=\frac{u_n}{\|u_n\|_{H^{1}_{V}(\mathbb{R}^3)}}$, $\{v_n\}_{n=1}^\infty$ is a bounded sequence in $H^{1}_{V}(\mathbb{R}^3)$. Up to subsequence, one has
\begin{equation*}
v_n\to v\ \mbox{ weakly in}\ H^{1}_{V}(\mathbb{R}^3),\ v_n\to v\ \mbox{ strongly in}\ L^{2}(\mathbb{R}^3),\ v_n\to v\ \mbox{a.e. in}\ \mathbb{R}^3.
\end{equation*}
Multiplying both side of \eqref{eq27} by $\|u_n\|^{-2}_{H^{1}_{V}(\mathbb{R}^3)}$, we obtain\\
\begin{equation*}
\left(\frac{1}{2}-\frac{1}{3}\right)\int_{\mathbb{R}^3}m v^{2}_n\text{d}x\geq\frac{1}{6}.
\end{equation*}
From $v_n\to v$ strongly in $L^{2}(\mathbb{R}^3)$, it follows that $v\neq0$. Thus the Lebesgue measure of the set
$$\mathcal{M}=\{x\in\mathbb{R}^3:v(x)\neq 0\}$$
is positive.

For $x\in\mathcal{M}$, as $n\to\infty$, we deduce
$$v_n(x)\to v(x)\neq 0~\text{and}~u_{n}(x)=v_{n}(x)\|u_n\|_{H^1_{V}(\mathbb{R}^3)}\to\infty.$$
It follows from the Fatou lemma that
\begin{equation*}
	\begin{split}
&\indent\int_{\mathbb{R}^3}\left(\phi_{u_n}u^{2}_{n}+mu_{n}^{2}\right)\|u_n\|^{-2}_{H^{1}_{V}(\mathbb{R}^3)}\text{d}x\\
&\geq\int_{\mathbb{R}^3}\frac{\phi_{u_n}u^{2}_{n}}{u_{n}^{2}}v^{2}_{n}\text{d}x\\
&\geq \int_{\mathcal{M}}\phi_{u_n}v^{2}_n\text{d}x\to\infty,
\end{split}
\end{equation*}
as $n\to\infty$. Thus
\begin{equation*}
		\begin{split}
		0&\leq \langle I_{\epsilon,R}'(u_n),u_n\rangle\\
		&=\int_{\mathbb{R}^3}\left(|\nabla u_n|^{2}+\tilde V(x)u_n^{2}\right)\text{d}x-\int_{\mathbb{R}^3}\phi_{u_n}u_n^{2}\text{d}x-m\int_{\mathbb{R}^3}u_n^{2}\text{d}x\\
		&\leq\|u_n\|_{H^{1}_{V}(\mathbb{R}^3)}^{2}\left(1-\int_{\mathcal{M}}\phi_{u_n}v^{2}_n\text{d}x\right)\to-\infty,
	\end{split}\end{equation*}
as  $n\to\infty$. This is impossible. Therefore, our claim is true.

Due to this claim and \eqref{eq25}, for $B>A$ large enough, there exists an unique $T:=T(u)>0$ such that
$$I_{\epsilon,R}(T(u)u)=-B\ \ \text{for every }u\in S^{\infty}.$$
It follows from the implicit function theorem that
\begin{center}
$T$ is a continuous function from $S^{\infty}$ to $\mathbb{R}$.
\end{center}
Thus the deformation retract $\eta:[0,1]\times(H^{1}_{V}(\mathbb{R}^3)\setminus B_{1})\to H^{1}_{V}(\mathbb{R}^3)$ denoted as
$$\eta(s,u)=(1-s)u+sT(u)u$$
with $\eta(0,u)=u$, $\eta(1,u)\in I_{\epsilon,R}^{-B}$. Thus
$$C_{q}(I_{\epsilon,R},\infty)=H^{q}(H^{1}_{V}(\mathbb{R}^3),I_{\epsilon,R}^{-B})\cong H^{q}(H^{1}_{V}(\mathbb{R}^3),H^{1}_{V}(\mathbb{R}^3)\setminus B_{1})\cong0,\ \mbox{for all}\ q\in \mathbb{N}.$$
This completes the proof of  Lemma \ref{lem6}.
\end{proof}
\noindent{\bf Proof of Theorem \ref{thm2}} It follows from Lemma \ref{lem4} that $I_{\epsilon,R}$ satisfies the Cerami conditions. By Lemma \ref{lem5}, $I_{\epsilon,R}$ has a local linking at zero with respect to the decomposition $H^{1}_{V}(\mathbb{R}^3)=H^{1}_{V}(\mathbb{R}^3)^{-}\oplus H^{1}_{V}(\mathbb{R}^3)^{+}$, thus from Proposition \ref{pro2}, for $d=\dim H^{1}_{V}(\mathbb{R}^3)^{-}$, it implies $C_{d}(I_{\epsilon},0)\neq0.$ On the other hand, due to Lemma \ref{lem6}, we have for all $q\in\mathbb{N}$, $C_{q}(
I_{\epsilon,R},\infty)=0$. Thus, by Proposition \ref{pro1}, for any $R>1$, there exists $\epsilon^{\ast}(R)=\min\{\epsilon(R),\epsilon''(R),\epsilon'''(R)\}$ for any $\epsilon\in(0,\epsilon^{\ast}(R))$ such that  $I_{\epsilon,R}$ has a nontrivial critical point $u_{\epsilon,R}$ and $I_{\epsilon,R}(u_{\epsilon,R})<0$.

By the proof of Lemma \ref{lem3.5}, for any $R>1$ fixed, one has
\begin{equation*}
		\begin{split}
		&\lim\limits_{\epsilon\to0^+}\left(I_{\epsilon,R}(u_{\epsilon,R})-I_{0}(u_{\epsilon,R})\right)\\
        =&\lim\limits_{\epsilon\to0^+}\frac{1}{\epsilon ^{2}}\eta _{R}\left (\|u_{\epsilon,R}\|^{2}_{H^1_{V}(\mathbb{R}^3)}\right )\int_{\mathbb{R}^{3}}F\left (\epsilon \eta _{R} \left ( \left | u_{\epsilon,R}\right |\right )u_{\epsilon,R}\right )\text{d}x\\
        =&0.
	\end{split}\end{equation*}
Similarly, for any $v\in H^{1}_{V}(\mathbb{R}^3)$, it implies
\begin{equation*}
		\begin{split}
		&\lim\limits_{\epsilon\to0^+}\langle I_{\epsilon,R}'(u_{\epsilon,R})-I_{0}'(u_{\epsilon,R}),u_{\epsilon,R}+v\rangle\\
        =&\lim\limits_{\epsilon\to0^+}\frac{2}{\epsilon^{2}}\eta _{R}'\left ( \|u_{\epsilon,R}\|^{2}_{H^1_{V}(\mathbb{R}^3)} \right )\int_{\mathbb{R}^{3}}(|\nabla u_{\epsilon,R}|^2+\tilde V(x)u^2_{\epsilon,R}) \text{d}x\int_{\mathbb{R}^{3}}F\left (\epsilon\eta _{R}\left ( \left | u_{\epsilon,R}\right | \right ) u_{\epsilon,R}\right )\text{d}x\\
        &+\lim\limits_{\epsilon\to0^+}\frac{1}{\epsilon}\eta _{R}\left ( \|u_{\epsilon,R}\|^{2}_{H^1_{V}(\mathbb{R}^3)} \right )\int_{\mathbb{R}^{3}}\left(f\left (\epsilon \eta _{R}\left ( \left | u_{\epsilon,R}\right | \right ) u_{\epsilon,R}\right )\left (\eta _{R}\left ( \left | u_{\epsilon,R}\right | \right ) u_{\epsilon,R} +\eta _{R}'\left ( \left | u_{\epsilon,R}\right |\right )|u_{\epsilon,R}|u_{\epsilon,R} \right )\right)\text{d}x\\
        &+\lim\limits_{\epsilon\to0^+}\frac{2}{\epsilon^{2}}\eta _{R}'\left ( \|u_{\epsilon,R}\|^{2}_{H^1_{V}(\mathbb{R}^3)} \right )\int_{\mathbb{R}^{3}}(\nabla u_{\epsilon,R}\nabla v+\tilde V(x)u_{\epsilon,R}v) \text{d}x\int_{\mathbb{R}^{3}}F\left (\epsilon\eta _{R}\left ( \left | u_{\epsilon,R}\right | \right ) u_{\epsilon,R}\right )\text{d}x\\
        &+\lim\limits_{\epsilon\to0^+}\frac{1}{\epsilon}\eta _{R}\left ( \|u_{\epsilon,R}\|^{2}_{H^1_{V}(\mathbb{R}^3)} \right )\int_{\mathbb{R}^{3}}\left(f\left (\epsilon \eta _{R}\left ( \left | u_{\epsilon,R}\right | \right ) u_{\epsilon,R}\right )\left (\eta _{R}\left ( \left | u_{\epsilon,R}\right | \right ) v +\eta _{R}'\left ( \left | u_{\epsilon,R}\right |\right )|u_{\epsilon,R}|v \right )\right)\text{d}x\\
        =&0.
	\end{split}\end{equation*}
Thus when $\epsilon\to0^{+}$, $\{u_{\epsilon,R}\}$ is a Cerami sequence of $I_0$. Then by the Cerami condition of $I_0$, we have $u_{\epsilon,R}\to u_0$ strongly in $H^1_{V}(\mathbb{R}^3)$ as $\epsilon\to0^+$, moreover $\{u_{\epsilon,R}\}$ is uniformly bounded in $H^1_{V}(\mathbb{R}^3)$. Similar with Lemma \ref{lem3}, we can obtain the $L^{\infty}$ norm of $u_{\epsilon,R}$ is uniformly bounded about $\epsilon$ and $R$. This means that there exists a positive constant $M$ independent of $\epsilon$ and $R$, such that for any $\epsilon\in(0,\epsilon^{\ast}(R))$ one has
$$\max\{\|u_{\epsilon,R}\|^2_{H^1_{V}(\mathbb{R}^3)},\|u_{\epsilon,R}\|_{L^{\infty}(\mathbb{R}^3)}\}\leq M.$$

Finally, let $R=M+1$, for any $\epsilon\in(0,\epsilon^{\ast}(M+1))$, then
\begin{equation*}
\begin{split}
    I_{\epsilon,M+1}\left ( u_{\epsilon,M+1} \right )&=\frac{1}{2}\int_{\mathbb R^3}\left(|\nabla u_{\epsilon,M+1}|^2+V(x)u_{\epsilon,M+1}^2\right)\text{d}x-\frac{1}{4}\int_{\mathbb R^3}\phi_{u_{\epsilon,M+1}}u_{\epsilon,M+1}^2\text{d}x\\
    &\quad-\frac{1}{\epsilon^2}\eta _{M+1}\left (\|u_{\epsilon,M+1}\|^{2}_{H^1_{V}(\mathbb{R}^3)}\right )\int_{\mathbb{R}^{3}}F\left (\epsilon \eta _{M+1} \left ( \left | u_{\epsilon,M+1}\right |\right )u_{\epsilon,M+1}\right )\text{d}x\\
    &=\frac{1}{2}\int_{\mathbb{R}^3}\left(|\nabla u_{\epsilon,M+1}|^{2}+V(x)u_{\epsilon,M+1}^2\right)\text{d}x-\frac{1}{4}\int_{\mathbb{R}^{3}}\phi _{u_{\epsilon,M+1}}u_{\epsilon,M+1}^{2}\text{d}x\\
    &\quad-\frac{1}{\epsilon ^{2}}\int_{\mathbb{R}^{3}}F\left (\epsilon u_{\epsilon,M+1}\right )\text{d}x
    \end{split}
\end{equation*}
has a nontrivial critical point $u_{\epsilon,M+1}$.  Hence, denote $\epsilon^{**}=\epsilon^{\ast}(M+1)$, then for any $\epsilon\in(0,\epsilon^{**})$, $u_{\epsilon,M+1}$ is also the solution of \eqref{eq2}. Consequently, $\epsilon u_{\epsilon,M+1}$ is the solution of $\eqref{eq1}$ and $\lambda^{**}=\frac{1}{(\epsilon^{**})^{2}}$. Thus Theorem \ref{thm2} has been proved.

\section{The proof of Theorem \ref{thm3}}

Now we study the asymptotical behavior of non-trivial solutions in Theorem \ref{thm}.

\noindent{\bf Proof of Theorem \ref{thm3}} In order to highlight the parameter dependence with respect to $\epsilon$, we rewrite $u_\epsilon$ denoting the
 solution of \eqref{eq7} and $c_\epsilon$ denotes the critical value of $I_{\epsilon}(u_\epsilon)$. By Theorem \ref{thm}, for any $0<\epsilon<\epsilon^{*}$, $u_{\epsilon}$ is also the solution of system \eqref{eq2}. Consequently, $\epsilon u_{\epsilon}$ is the solution of system $\eqref{eq1}$.
Clearly, $\{u_{\epsilon}\}$ is a bounded sequence in $H^{1}_{r}(\mathbb{R}^3)$.

Recalling \eqref{I0} and \eqref{eq5.1},
it is easy to say $I_0\in C^{1}$ in $H^{1}_{r}(\mathbb{R}^3)$.  Since $H^1_r(\mathbb R^3)$ is compactly embedded in $L^p(\mathbb{R}^3)(2<p<6)$, $I_0$ satisfies the (PS) condition.

By \eqref{eq6}, one has
\begin{equation*}
		\begin{split}
		\lim\limits_{\epsilon\to0}\left(I_{\epsilon}(u_{\epsilon})-I_{0}(u_{\epsilon})\right)=	\lim\limits_{\epsilon\to0}\frac{1}{\epsilon^2}\int_{\mathbb R^3}\tilde F(\epsilon u_{\epsilon})\text{d}x=0.
	\end{split}\end{equation*}
Similarly, for any $v\in H^{1}_{r}(\mathbb{R}^3)$, it implies
\begin{equation*}
		\begin{split}
		\lim\limits_{\epsilon\to0}\langle I_{\epsilon}'(u_{\epsilon})-I_{0}'(u_\epsilon),v\rangle=	\lim\limits_{\epsilon\to0}\frac{1}{\epsilon}\int_{\mathbb R^3}\tilde f(\epsilon u_{\epsilon})v\text{d}x=0.
	\end{split}\end{equation*}
Thus $\{u_\epsilon\}$ is a (PS) sequence of $I_0$. Then by the (PS) condition of $I_0$, we have
$$u_{\epsilon}\to u_0~\text{strongly~in}~H^{1}_{r}(\mathbb{R}^3)\ ~\text{as}~\epsilon\to0,$$
where $u_0$ is a critical point of $I_0$ and it is also a weak solution of equation \eqref{eq5.1}. Recall that $\lambda=\frac{1}{\epsilon^2}$ and $u^{\lambda}=\epsilon u_{\epsilon}$, hence Theorem \ref{thm3} has been proved.

	\medskip
	\textbf{Acknowledgement.}
	C. Huang was supported by China Postdoctoral Science Foundation (No.2020M682065). S. Liang was supported by  the Young Outstanding Talents Project of Scientific Innovation and entrepreneurship in Jilin (No. 20240601048RC), the National Natural Science
Foundation of China (No. 12371455). L. Ma was supported by National Natural Science Foundation of China (No. 12301288). P. Pucci is a member of the  Gruppo Nazionale per
l'Analisi Ma\-te\-ma\-ti\-ca, la Probabilit\`a e le loro Applicazioni (GNAMPA) of the  Instituto Nazionale di Alta Matematica (INdAM)
and this paper was written under the auspices of GNAMPA--INdAM.

	\nocite{*}
	\bibliographystyle{abbrv}
	\bibliography{ref}
\end{document}